\documentclass[11pt,a4paper]{article}
\usepackage[T1]{fontenc}
\usepackage[utf8]{inputenc}
\usepackage{lmodern,microtype}
\usepackage[margin=29mm]{geometry}
\usepackage{amsmath,amssymb,amsthm,mathtools}
\usepackage{booktabs,array,enumitem,xurl}
\usepackage[colorlinks=true,linkcolor=blue,citecolor=blue,urlcolor=blue]{hyperref}

\numberwithin{equation}{section}
\newtheorem{theorem}{Theorem}[section]
\newtheorem{proposition}[theorem]{Proposition}
\newtheorem{lemma}[theorem]{Lemma}
\newtheorem{corollary}[theorem]{Corollary}

\theoremstyle{definition}

\newcommand{\one}{\mathbf 1}
\newcommand{\E}{\mathbb E}

\newcommand{\Sg}{\mathfrak S}
\newcommand{\Res}{\operatorname{Res}}

\title{Shifted convolutions of divisor functions}
\author{Miguel N. Walsh}
\date{}
\begin{document}
\maketitle
\begin{abstract}
We establish the leading asymptotic for the shifted correlations $\sum_{n \le x} d_k(n)d_l(n+h)$ of divisor functions.
\end{abstract}
\tableofcontents

\section{Introduction}
\label{sec:introduction}

For an integer $k\ge1$, let $d_k(n)$ be the standard divisor function counting the ordered factorizations
of $n$ into $k$ positive factors.
The purpose of this article is to establish the following result.

\begin{theorem}\label{thm:main}
Given integers $k,l\ge1$, $h\ne0$, there exists $\eta=\eta(k,l)>0$
such that, for every $\Phi\in C_c^\infty((0,\infty))$,
\begin{equation}\label{eq:smooth-main}
\begin{aligned}
 \sum_n\Phi(n/x)d_k(n)d_l(n+h)
 &=\frac{\Sg_{k,l}(h)}{(k-1)!(l-1)!}
       x(\log x)^{k+l-2}\int_0^\infty\Phi(u)\,du\\
 &\quad+O_{k,l,h,\Phi}\bigl(x(\log x)^{k+l-2-\eta}\bigr),
\end{aligned}
\end{equation}
where the singular series is given by the constant
\begin{equation}\label{eq:singular}
 \Sg_{k,l}(h)=\prod_p(1-p^{-1})^{k+l-2}I_{p,h},\qquad
 I_{p,h}=\int_{\mathbb Z_p}
       d_k(p^{v_p(u)})d_l(p^{v_p(u+h)})\,du,
\end{equation}
with the normalized Haar measure over $\mathbb{Z}_p$. In particular, for every $h, k, l \ge 1$ we have
\begin{equation}\label{eq:sharp-main}
 \sum_{1\le n\le x}d_k(n)d_l(n+h)
 \sim\frac{\Sg_{k,l}(h)}{(k-1)!(l-1)!}
                 x(\log x)^{k+l-2}.
\end{equation}
\end{theorem}
Theorem~\ref{thm:main} is made possible by the weighted-dilation graph estimates
of \cite{OpenAIGraph}, which are in turn based on the divisibility-graph framework developed by Helfgott and Radziwi{\l}{\l}~\cite{HR} and further refined by Pilatte~\cite{Pilatte}. We will also take advantage of the zero-free region for
Dirichlet $L$-functions given in \cite{OpenAIZeroFree}, although this is not strictly necessary and the same result could be recovered using classical estimates. The methods are rather general and can be extended to other shifted convolution problems, as discussed in Section~\ref{subsec:further-results}.

\subsection{Prior work}

Estimating the shifted correlations of divisor functions is a classical
problem in analytic number theory, with its expected asymptotics being conjectured in several places (e.g.
\cite{ConreyGonek2001,Ivic1997General,TaoDivisorHeuristic,Titchmarsh1942,Vinogradov1988}).
For $k=l=2$, the leading asymptotic was obtained in $1927$
by Ingham~\cite{Ingham1927}, with Estermann~\cite{Estermann1931}
later providing the full quadratic main term
(see also~\cite{DeshouillersIwaniec1982,Motohashi1994}).

For $k=2$, $l=3$, the leading asymptotic was established
by Hooley~\cite{Hooley1957}. Linnik~\cite{Linnik1961}
extended this to $k=2$ and arbitrary $l$, initially for
the shift $h=1$, while Motohashi~\cite{Motohashi1976,Motohashi1980}
subsequently obtained the full polynomial asymptotic in this case for
arbitrary fixed shifts. Taking $h=1$, Fouvry and
Tenenbaum~\cite{FouvryTenenbaum1985} improved the error
term for $k=2$, $l\ge4$, while
Heath-Brown~\cite{HeathBrown1986} obtained a power-saving
error for $k=2$, $l=3$.

No $k, l \ge 3$ case of Theorem \ref{thm:main} was previously known.
In~\cite{MRTCorrelationsI,MRTCorrelationsII},
Matom\"aki, Radziwi{\l}{\l} and Tao showed that the
leading asymptotic in Theorem~\ref{thm:main} holds for
almost all shifts in ranges as short as a fixed power
of $\log x$
(see also~\cite{BBMZ,IvicWu2012,PliegoSunWang2025}).
Further work on these and closely related problems
can be found in~\cite{AndradeSmith,BettinBuiLiRadziwill2020,ConreyKeating2015III,ConreyKeating2019,Drappeau2017,DrappeauTopacogullari2019,DukeFriedlanderIwaniec1994,HamiehVo2026,JiangLu2022,Lau2025v2,NgThom2019,Nguyen2023Correlation,SmithThesis,Topacogullari2016,Topacogullari2018,Wang2024}.

\subsection{Outline of the proof}
\label{subsec:outline-of-proof}

We now briefly describe the main ideas behind the proof. 

\subsubsection{The density model and invariant factors}
\label{subsubsec:outline-reduction}

As in previous work~\cite{AndradeSmith,MRSTT,MSTT,SmithThesis, Wang2024},
we begin by approximating $d_k(n)$ by a suitable model.
Put $L=\log x$ and let $\mathcal S$ be the set of primes
of size at most $\exp(L^{2/5})$.
We will consider a density model $D_k$ for the divisor
function which retains the dependence of $d_k(n)$
on the primes in $\mathcal S$, but replaces its
dependence on the remaining primes by the corresponding
mean-value density. Since the primes in $\mathcal S$
are relatively small, the shifted correlation
of $D_k(n)$ with $d_l(n+h)$ can be evaluated by
standard techniques. This is carried out in
Proposition~\ref{prop:scalar-proxy}.

The main task is therefore to control the correlation
of the centered remainder $\xi_k(n)=d_k(n)-D_k(n)$ with $d_l(n+h)$.
The crucial feature of the way we have chosen our density model is that, for any set of primes
$\mathcal P\subseteq\mathcal S$, we can write
\[
\xi_k(n)=d_k(n_{\mathcal P})\alpha_{\mathcal P}(n),
\qquad
d_l(n)=d_l(n_{\mathcal P})\beta_{\mathcal P}(n),
\]
where $n_{\mathcal P}$ is the largest divisor of $n$ supported on primes in $\mathcal P$ and the functions $\alpha_{\mathcal P}$ and
$\beta_{\mathcal P}$ are invariant under multiplication
by such primes.

If we could simply discard the factors depending on
$n_{\mathcal P}$, the problem would reduce to controlling
the shifted correlation
$\alpha_{\mathcal P}(n)\beta_{\mathcal P}(n+h)$.
One could then hope to exploit the dilation invariance
of these functions to generate enough symmetries to
obtain cancellation, as we will discuss in the next
subsection. Of course, we cannot discard these factors
directly and, in fact, the unbounded nature of the divisor functions
could make their contribution particularly problematic.

We overcome this difficulty by restricting attention
to integers having exactly one distinct prime divisor
in certain intervals of primes, called
\emph{banks}. These banks consist of primes of size
$\exp(L^a)$, for suitable values of $a$, and have
reciprocal-prime mass $\sum_{p\in P_g}1/p\asymp1$.
For each bank, we expect with positive probability that
both $n$ and $n+h$ have exactly one prime divisor
in it. Since there are $\asymp\log\log x$ such banks,
we expect a typical integer to satisfy this condition
in a positive proportion of them.

However, it is not enough to show that exceptions
are rare, since the divisor functions could certainly
concentrate on integers with atypical prime factorizations.
To control this possibility, we use upper bounds of
Henriot for multiplicative functions evaluated at
affine forms~\cite{Henriot,HenriotErr}. These allow
us to construct an enveloping measure that controls
divisor-weighted correlations subject to arbitrary
conditions on their prime factors in the banks (Lemma~\ref{lem:positive-record-law}).
In particular, we will show that integers with too few
banks satisfying the desired condition make a negligible
contribution to our correlation (Corollary~\ref{cor:good-records}).

Once this is established, an inclusion-exclusion
argument allows us to reduce the problem to
correlations restricted to integers $n$ for which
both $n$ and $n+h$ have exactly one distinct prime
divisor in each bank of a sufficiently large
collection. Let $\mathcal P$ denote the union of
these banks. After discarding the negligible
contribution of higher prime powers, each selected
bank contributes exactly the value $kl$ to
$d_k(n_{\mathcal P})d_l((n+h)_{\mathcal P})$.
These divisor factors are therefore
constant on the restricted set and the problem
effectively reduces to controlling the shifted
correlation
$\alpha_{\mathcal P}(n)\beta_{\mathcal P}(n+h)$ (see Section~\ref{sec:completion}).

\subsubsection{The dilation graphs}
\label{subsubsec:outline-graph}

Since Tao's work on the two-point logarithmically
averaged Chowla conjecture~\cite{TaoLogChowla},
it is natural to try to exploit the dilation invariance
of $\alpha_{\mathcal P}$ and $\beta_{\mathcal P}$
by averaging over primes in $\mathcal P$. Writing $V_{\mathcal P}=\sum_{p\in\mathcal P}1/p$,
we have
\[
\sum_{n\le x}
\frac{\alpha_{\mathcal P}(n)\beta_{\mathcal P}(n+h)}n
=
\frac1{V_{\mathcal P}}
\sum_{p\in\mathcal P}
\sum_{\substack{n\le px\\p\mid n}}
\frac{\alpha_{\mathcal P}(n)
\beta_{\mathcal P}(n+hp)}n.
\]
This replaces the fixed shift $h$ by an average over
shifts $hp$, but restricted to multiples of $p$.
In Tao's argument, an entropy-decrement method
allows this divisibility restriction to be essentially removed when the primes are sufficiently small,
making the resulting averages accessible to more
classical methods. However, the primes are required to be of sublogarithmic size and the savings one can thus obtain are rather modest.

Helfgott and Radziwi{\l}{\l}~\cite{HR} introduced
an alternative approach based on high-trace estimates
for the corresponding graph operator. This allows the
primes to be much larger, of size up to
$\exp(L^{1/2-\varepsilon})$ and thus also opens the door to much stronger savings.
In particular, Pilatte~\cite{Pilatte} extended these
methods to allow shifts made out of composite $D$. In his work, these are products of primes chosen from
many of the banks discussed above. Loosely speaking,
there are $\asymp\log\log x$ such banks, each providing
a fixed saving, leading to a bound of order $L^{-c}$
for some $c>0$.

Unfortunately, unlike in the case of bounded
multiplicative functions, our reduction of the divisor
problem to correlations of $\alpha_{\mathcal P}$
and $\beta_{\mathcal P}$ incurs losses of several
powers of $L$. The savings from the
previous methods are therefore insufficient for
our purposes, even if we were content with a
logarithmically averaged result. We instead need
bounds of order $O_A(L^{-A})$ for arbitrarily large
fixed $A$. These stronger savings are made available by the graph
machinery of \cite{OpenAIGraph}, which builds
largely on the methods of \cite{HR}.
Its main innovation is to make it possible to work
with sparse sets $\Omega$ of integers having an
unusually large number of distinct prime divisors
in each bank and hence many more possible shifts
$hD$ from which to obtain mixing.

Working with these sparse sets introduces difficulties
absent from the earlier arguments. For example,
even when $D$ is a product of prime divisors of
$n\in\Omega$, the translated integer $n+hD$ need
not have enough additional prime divisors to remain
in $\Omega$. The approach in \cite{OpenAIGraph}
overcomes such difficulties through an additional argument in
the spirit of \cite{HR,Pilatte}, showing that, after
averaging over $n$ and $D$, the divisibility conditions
required for $n+hD$ to belong to $\Omega$ behave
sufficiently like a random model that draws primes
freely from the banks.

This transference is costly, but the crucial observation
of \cite{OpenAIGraph} is that we do not need to free
all the primes in $D$ from their divisibility restrictions to obtain cancellation. Freeing only a suitably chosen
number from each bank already gives arbitrarily
strong logarithmic savings, while
keeping the cost of the comparisons under control.
Roughly speaking, this is made possible by exploiting the averaging in the choice 
of which of the many prime divisors from each bank to free in a given edge.

These stronger savings also allow us to absorb the
losses arising from the different scales reached by
dilation, avoiding the need for logarithmic averaging. Unfortunately, 
our setting is slightly different from that of
\cite{OpenAIGraph} and
we must verify that their results apply to the particular
graphs required by our bank restrictions.
We carry out this argument in Section~\ref{sec:graph}.

\subsubsection{Estimating the comparison correlations}
\label{subsubsec:outline-completion}

Writing $\mathcal C_0$ for the fixed-shift correlation
involving $\alpha_{\mathcal P}(n)\beta_{\mathcal P}(n+h)$
that we wish to estimate, the graph argument yields a relation
of the form
\[
 \mathcal C_0+\sum_\nu c_\nu\mathcal C_\nu
       =O_A(L^{-A}),
\]
where the number of comparison terms and the total absolute
mass of their real coefficients $c_\nu$ are bounded by a
fixed power of $L$ (Proposition~\ref{prop:graph-comparison}). Here, suppressing smooth cutoffs and
additional weights, each $\mathcal C_\nu$ is
essentially of the form
\[
\mathbb E_{D_1,D_2} b_1(D_1)b_2(D_2)
\sum_n \alpha_{\mathcal P}(n)
\beta_{\mathcal P}(n+hD_1D_2),
\]
for certain bounded functions $b_1,b_2$.
Crucially, the fixed shift $h$ has been replaced by
an average over shifts $hD_1D_2$, where $D_1,D_2$
are independently chosen products of primes from
certain banks. This additional freedom makes
cancellation accessible through a more standard
circle-method argument similar to that of
\cite[Sections~5.3--5.6]{OpenAIGraph}.

Indeed, after Fourier inversion, we obtain a
multiplier involving bilinear averages of the form
\[
 \mathbb E_{D_1,D_2}
 b_1(D_1)b_2(D_2)e(\theta D_1D_2),
 \qquad |b_1|,|b_2|\le1,
\]
where $e(z)=\exp(2\pi iz)$.
The presence of two independent products allows us
to control this multiplier uniformly over the
bounded functions $b_1,b_2$, away from neighborhoods
of rationals with small denominators (Lemma~\ref{lem:auxiliary-product}).

The analysis of the resulting Fourier integral is
therefore reduced to the major arcs. There, Cauchy–Schwarz and divisor moments allow us to control the $\beta_{\mathcal P}$ factor, leaving us to exploit cancellation in $\alpha_{\mathcal P}$. The problem is then
essentially reduced to understanding twisted
variants of its Dirichlet polynomials
$\sum_n\alpha_{\mathcal P}(n)n^{it}$ (Lemma~\ref{lem:local-mellin-energy}).

For small values of $|t|$, we recall that $\alpha_{\mathcal{P}}$ arises from the difference between the divisor function and our density model $D_k$. After fixing
the small-prime part of $n$, the mean-value
contributions of the divisor function and its
model have exactly the same principal residue.
Subtracting the model therefore cancels the
entire principal contribution, yielding the
required cancellation at low frequencies (Lemma~\ref{lem:combined-mellin}).

For large values of $|t|$, we instead write
$n=pm$, with $p$ ranging over primes in a small
bank. We estimate the resulting prime sums using
the zero-free region of \cite[Theorem~1.1]{OpenAIZeroFree}
(see Lemma~\ref{lem:zf-smooth-prime}), while the remaining factors
are controlled by a mean-square estimate.
This is in the spirit of \cite{MR} (see also
\cite[Section~3]{SoundararajanBourbaki}),
although our restriction to banks containing
exactly one distinct prime divisor makes
the prime-factor decomposition almost immediate.
The auxiliary products $D_1,D_2$ are sufficiently
large compared with the extracted prime
to ensure that the remaining Dirichlet polynomials
are long enough for the required mean-square estimates.

Section~\ref{sec:analytic-services} establishes
the basic analytic estimates needed for this
major-arc argument, which is carried out in
Section~\ref{sec:arithmetic-bridge}.
The resulting savings are stronger than
any prescribed negative power of $L$ 
and therefore suffice to control the full
family of comparison terms, finishing the proof of
Theorem~\ref{thm:main}.

\subsubsection{Further results}
\label{subsec:further-results}

The strategy just described is rather general and should adapt
to several other shifted-correlation problems, including
$\lambda_f(n)^2\lambda_g(n+h)^2$ for normalized Hecke eigenvalues
of fixed holomorphic cusp forms
(see~\cite[Section~1.3.1]{Mangerel2023}) and correlations of the indicator
of integers representable as sums of two squares
(as conjectured in~\cite[Section~3]{ConnorsKeating1997} and
\cite[Conjecture~2]{Smilansky2013}). In particular, each of these problems admits a
density model for which both the model and the associated
remainder are amenable to estimates analogous to those established in
the present article. The same strategy also appears robust enough to
give uniformity for $0<|h|\ll x^{1-\varepsilon}$ and to refine
Theorem~\ref{thm:main} by recovering the expected polynomial main term
with error
$O_\varepsilon(x(\log x)^{\max\{k,l\}-2+\varepsilon})$.
We invite the interested reader to pursue these extensions.

\subsection*{AI use}

A personal ChatGPT Pro subscription was used both in the exploratory research phase and during the writing process. 

\section{Preliminaries}
\label{sec:records}

We begin by recording some conventions, definitions and basic estimates that will be used throughout the article.

\subsection{The general setting} We write $F\ll G$ or $F=O(G)$ if $|F|\le C|G|$ for some
constant $C>0$. We write $F\asymp G$ if both $F\ll G$ and $G\ll F$.
Dependence on parameters is indicated by subscripts, but
dependence on the fixed divisor orders $k,l$, shift $h$ and smooth
weights is usually suppressed. We write $F=o(G)$ if $F/G\to0$
and $F\sim G$ if $F/G\to1$, as $x\to\infty$.

Throughout the article, we will write $L=\log x$ and $J=\log L$. We also write $e(t)=\exp(2\pi i t)$ and
$(a)_r=a(a-1)\cdots(a-r+1)$, with $(a)_0=1$.
If $\mathcal U$ is a finite set of primes, let
\[
 n_{\mathcal U}=\prod_{p\in\mathcal U}p^{v_p(n)},\qquad
 n_{\mathcal U'}=n/n_{\mathcal U}.
\]
We will call the first quantity the \emph{$\mathcal U$-record} of $n$ and sometimes omit mentioning $\mathcal{U}$ explicitly when it is clear from the context. In particular, records include prime powers.
We shall refer to integers supported on $\mathcal U$ as
$\mathcal U$-smooth.

During the arguments we will concentrate on integers having a prescribed number of distinct prime divisors in certain intervals. The set of primes in each of these intervals will be called \emph{a bank}. More precisely, it will suffice to choose two families $\mathcal G_s, \mathcal G_b$ of pairwise disjoint "small" and "big" banks $P_g$ given by intervals of primes satisfying:
\begin{equation}\label{eq:banks}
 \begin{aligned}
 P_g&\subset[\exp(L^{1/8}),\exp(L^{1/4})]
       &&(g\in\mathcal G_s),\\
 P_g&\subset[\exp(L^{1/3}),\exp(L^{2/5})]
       &&(g\in\mathcal G_b),\\
 \tfrac12&\le V_g:=\sum_{p\in P_g}p^{-1}\le2,
 &c_*J&\le |\mathcal G_s|,|\mathcal G_b|\le C_*J.
 \end{aligned}
\end{equation}
Here $c_*,C_*>0$ are absolute constants. For example, by Mertens' theorem we may take $\mathcal G_s, \mathcal G_b$ to index the sets of primes $P_g$ satisfying
$j<\log\log p\le j+1$ for a given integer $j$, whenever these are fully contained in the ranges given by  \eqref{eq:banks}.

We shall use the corresponding notations:
\begin{equation}\label{eq:fullrecord}
 y_0=\exp(L^{2/5}),\qquad \mathcal S=\{p:p\le y_0\},\qquad
 \mathcal C=\bigcup_{g\in\mathcal G_s\cup\mathcal G_b}P_g,\qquad
 V=\prod_{p\le y_0}(1-p^{-1}).
\end{equation}
We will frequently be working with sets $\mathcal P$ consisting of all the primes that belong to specific choices of banks. The selected banks will be called \emph{active} and $\mathcal P$ the active set of primes.
Thus $\mathcal P\subset\mathcal C\subset\mathcal S$.
We shall write $\mathcal G_{\mathcal P}=\{g:P_g\subset\mathcal P\}$ and
$\omega_g(n)=\#\{p\in P_g:p\mid n\}$.

\subsection{The density model}
We now describe the model we will use to approximate the divisor function. Writing $d_k(n) = d_k(n_{\mathcal S}) d_k(n_{\mathcal S'})$, we shall keep the factor depending on $\mathcal S$ and replace $d_k(n_{\mathcal S'})$ with its expected density, which we now proceed to estimate in a standard way (see e.g. \cite[Chapter~II.5]{Tenenbaum}). We consider the Dirichlet series over
integers avoiding $\mathcal S$,
\[
 \sum_{(n,\mathcal S)=1}\frac{d_k(n)}{n^s}
     =\zeta(s)^k\prod_{p\in\mathcal S}(1-p^{-s})^k
       \qquad(\Re s>1),
\]
and define the corresponding density:
\begin{equation}\label{eq:density-def}
 \tilde{M}_{k,\mathcal S}(t):=\Res_{s=1}
 \left\{\zeta(s)^k\prod_{p\in\mathcal S}(1-p^{-s})^k
                                    t^{s-1}\right\}.
\end{equation}
For a more explicit expression, setting $z=s-1$, rewriting the expression inside the residue as $z^{-k} H(z) e^{z \log t}$ with the entire function
\[
 H(z)=\left\{z\zeta(1+z)\prod_{p\in\mathcal S}
                                  (1-p^{-1-z})\right\}^{k}
\]
and looking at its residue through the Taylor expansion of its factors, we obtain
\begin{equation}\label{eq:density-finite}
 \tilde{M}_{k,\mathcal S}(t)=\sum_{j=0}^{k-1}
 \frac{H^{(j)}(0)}{j!}
 \frac{(\log t)^{k-1-j}}{(k-1-j)!},\qquad H(0)=V^k.
\end{equation}
The size restriction on the primes of $\mathcal S$ ensures that the corresponding Euler products do not distort the leading term. Indeed, writing $H(z) = V^k e^{F(z)}$ locally near $z=0$ and estimating the derivatives of $F(z)$ using
\begin{equation}\label{eq:record-derivatives}
 \sum_{p\le y_0}\frac{(\log p)^j}{p}\ll_j L^{2j/5}\qquad(j\ge1),
\end{equation}
we deduce that $|H^{(j)}(0)|\ll_{k}V^kL^{2j/5}$ for $0\le j\le k-1$, from which it follows that for every $x^{1/2}\le t\le x^2$, say, we have
\begin{equation}\label{eq:density-positive}
 \tilde{M}_{k,\mathcal S}(t)=
 \frac{V^k(\log t)^{k-1}}{(k-1)!}
       \{1+O_k(L^{-3/5})\}>0
      .
\end{equation}
Since the range $x^{1/2} \le t \le x^2$ amply covers the region where we will need the density model to coincide with the above residue, it suffices for us to define ${M}_{k,\mathcal S}(t) = \tilde{M}_{k,\mathcal S}(t)$ for $t\ge x^{1/2}$ and extend it to the range $0 < t < x^{1/2}$ in any smooth non-negative way that preserves, for $0<t\le x^2$, the bound 
\begin{equation}
\label{Mkbound}
 {M}_{k,\mathcal S}(t) \ll_k V^kL^{k-1}.
 \end{equation}
 For $k=1$, we may of course simply take $M_{1,\mathcal S}\equiv V$. 

We now define the model for $d_k$ and its corresponding remainder as
\begin{equation}\label{eq:model-def}
 D_k(n)=d_k(n_{\mathcal S})
              \frac{M_{k,\mathcal S}(n_{\mathcal S'})}{V},
 \qquad \xi_k(n)=d_k(n)-D_k(n).
\end{equation}
The factor $1/V$ normalizes the density $M_{k,\mathcal S}$
by the density $M_{1,\mathcal S}=V$ of integers avoiding
$\mathcal S$. Thus $M_{k,\mathcal S}/V$ models the mean
divisor weight among these integers. In particular, we have
$0\le D_k(n)\ll_k L^{k-1}d_k(n)$ for $n\le x^2$.

We record the following immediate consequence of \eqref{eq:density-positive} that shall be needed later.

\begin{lemma}\label{lem:density-size}
If $n\asymp x$ and $1 \le A\leq \exp(\sqrt{L})$, then
\begin{equation}\label{eq:density-leading}
 \frac{M_{j,\mathcal S}(n/A)}V
  =\frac{V^{j-1}L^{j-1}}{(j-1)!}\{1+O_j(L^{-1/2})\}.
\end{equation}
\end{lemma}

\subsection{The invariant factors}
Remember we will work with certain "active" sets of primes $\mathcal{P}$ consisting of the union of some of the prime banks we have defined. The key advantage we will derive from separating the factors of a given $n$ as $n=n_{\mathcal{P}}n_{\mathcal{P}'}$ is that the primes in $\mathcal{P}$ are sufficiently small for functions depending only on $n_{\mathcal{P}}$ to become tractable, while functions depending only on $n_{\mathcal{P}'}$ are then invariant under multiplication by elements of $\mathcal{P}$. In particular, for the latter we will be able to exploit the large number of symmetries caused by this invariance. With this in mind, for every such choice of $\mathcal{P}$ we define
\begin{equation}\label{eq:core-def}
 \begin{aligned}
 \alpha_{\mathcal P}(n)
   &=d_k(n_{\mathcal P'})
     -d_k(n_{\mathcal S\setminus\mathcal P})
                \frac{M_{k,\mathcal S}(n_{\mathcal S'})}{V},\\
 \beta_{\mathcal P}(n)&=d_l(n_{\mathcal P'}).
 \end{aligned}
\end{equation}
For every $p\in\mathcal P$, we then have $ \alpha_{\mathcal P}(pn)=\alpha_{\mathcal P}(n),
 \beta_{\mathcal P}(pn)=\beta_{\mathcal P}(n)$
and by multiplicativity
\begin{equation}\label{eq:core-factor}
 \xi_k(n)=d_k(n_{\mathcal P})\alpha_{\mathcal P}(n),\qquad
 d_l(n)=d_l(n_{\mathcal P})\beta_{\mathcal P}(n).
\end{equation}
Notice that taking $A=n_{\mathcal S}$ and $t=n/A$ we may simply write
\begin{equation}\label{eq:record-centered-core}
 \alpha_{\mathcal P}(n)=d_k(A_{\mathcal S\setminus\mathcal P})
                   \{d_k(t)-M_{k,\mathcal S}(t)/V\},
 \qquad (t,\mathcal S)=1.
\end{equation}

\subsection{Moments and truncation}
As in previous work (e.g. \cite{HR, MR, Pilatte}), integers with atypical factorizations in the original interval will be problematic and we will need some global control on them. Put
\[
 \theta=(\log y_0)^{-1}=L^{-2/5},\qquad E=\exp(\sqrt L).
\]
By Rankin's inequality, for each fixed $r\ge1$,
\[
 \mathbf1_{\{n_{\mathcal S}>E\}}
       \le E^{-r\theta}n_{\mathcal S}^{r\theta},\qquad
 E^{-r\theta}=\exp(-rL^{1/10}).
\]
We will need the following estimate that takes into account extra divisor weights.

\begin{lemma}[Elementary tilted moments]\label{lem:record-moments}
For fixed integers \(a,r\geq1\), fixed \(c>0\), fixed \(0<\rho<C\), and
\(x^\rho\leq N\leq x^C\), one has
\begin{equation}
 \sum_{n\leq cN}d_a(n)^r n_{\mathcal S}^{r\theta}
       \ll N L^{a^r-1}.
 \label{eq:tilted-moment}
\end{equation}
In particular,
\begin{equation}
\sum_{n\leq cN}d_a(n)^r
             \mathbf1_{\{n_{\mathcal S}>E\}}
       \ll N L^{a^r-1}\exp(-rL^{1/10}).
 \label{eq:full-record-tail}
\end{equation}
\end{lemma}

\begin{proof}
Recall the elementary estimate $d_a(n)^r\leq d_{a^r}(n)$ and write  \(A=a^r\) and \(f(n)=d_A(n)n_{\mathcal S}^{r\theta}\).
The prime-power values of $f$ are nondecreasing, so $g=f*\mu$ is
nonnegative, where $\mu$ is the M\"obius function. Therefore, by M\"obius inversion
\[
\sum_{n\leq cN} f(n)
 \leq cN\sum_{d\leq cN}\frac{g(d)}d
 \leq cN\prod_{p\leq cN}\left\{(1-p^{-1})
                         \sum_{\nu\geq0}\frac{f(p^\nu)}{p^\nu}\right\}.
\]
Notice that replacing $f$ with $d_A$ the product becomes $\prod_{p \le c N} (1-p^{-1})^{1-A} = O(L^{A-1})$, by Mertens' estimate.
But for sufficiently large \(x\), the logarithm of the ratio of the product we are interested in to that product is bounded by a constant times
\[
\sum_{p\in\mathcal S}\frac{p^{r\theta}-1}{p}
       +\sum_p p^{-3/2}
 \ll_r \theta\sum_{p\leq\exp(L^{2/5})}\frac{\log p}{p}+1
 \ll_r1.
\]
This proves \eqref{eq:tilted-moment}.  Multiplying by
\(E^{-r\theta}\) proves \eqref{eq:full-record-tail}.
\end{proof}

As a quick application, we record the following uniform weighted estimate for the moments of $\alpha_{\mathcal{P}}$, $\beta_{\mathcal{P}}$ which we shall also need later.

\begin{lemma}\label{lem:marked-moments}
Fix integers $M,r\ge1$ and an active set of primes $\mathcal P$. For each $n$, let $\mathcal L_M(n)$ be the (possibly empty) set of ordered
lists $\boldsymbol p=(p_{g,i})$ consisting of $M$ distinct primes
$p_{g,i}\in P_g$ dividing $n$ for each active bank $g$. Give $(n,\boldsymbol p)$ mass
$w_M=\prod_{g\in\mathcal G_{\mathcal P}}V_g^{-M}$. For
$N=x^{1+o(1)}$ and every $x$ sufficiently large with respect to $M$, we have
\[
 \sum_{n\asymp N}\ \sum_{\boldsymbol p\in\mathcal L_M(n)}
 w_M\bigl(|\alpha_{\mathcal P}(n)|^r+
                 |\beta_{\mathcal P}(n)|^r\bigr)
        \ll N L^{C_r},
\]
with $C_r$ and the implied constant independent of $M$, the choices of the lists and the active set of primes $\mathcal{P}$ (but permitted to depend on $r,k,l$).
\end{lemma}
\begin{proof}
Fix a list with corresponding product $P=\prod_{g,i}p_{g,i}$.
Since $P\mid n$ and all its primes are active, we see from (\ref{Mkbound}) that
$|\alpha_{\mathcal P}(n)|\ll L^{k-1}d_K(n/P)$ and
$|\beta_{\mathcal P}(n)|\le d_K(n/P)$ for
$K=\max(k,l)$. Here
$P=\exp(O_M(JL^{2/5}))=x^{o(1)}$, so Lemma~\ref{lem:record-moments}
at scale $N/P$ gives
$\sum_{n\asymp N,\,P\mid n}d_K(n/P)^r\ll(N/P)L^{K^r-1}$.
Thus, for the contribution of this particular list we could take $C_r=r(k-1)+K^r-1$.
Finally, summing over all lists costs at most
\[
\sum_{\text{lists}}\frac1P\prod_gV_g^{-M}
 \le\prod_g\left(\sum_{p\in P_g}\frac1{V_gp}\right)^M=1.
\]
\end{proof}

\section{Evaluation of the density model}
\label{sec:scalar-proxy}

This section will be devoted to the proof of the following correlation asymptotic for the density model (see also~\cite[Theorem~1.3]{AndradeSmith} and \cite[Theorem~3.2]{SmithThesis} for a similar statement regarding an alternative model).

\begin{proposition}\label{prop:scalar-proxy}
Fix integers $k,l\geq1$, a nonzero integer $h$, and
$\Phi\in C_c^\infty((0,\infty))$. Then
\begin{equation}
\begin{aligned}
 \sum_{n\geq1}\Phi(n/x)D_k(n)d_l(n+h)
 &=\frac{\mathfrak S_{k,l}(h)}{(k-1)!(l-1)!}
       xL^{k+l-2}\int_0^\infty\Phi(u)\,du\\
 &\quad+O_{k,l,h,\Phi}\!\left(xL^{k+l-5/2}\right),
 \end{aligned}
 \label{eq:scalar-proxy-asymptotic}
\end{equation}
for sufficiently large $x$.
Here $\mathfrak S_{k,l}(h)$ is the singular series defined in
\eqref{eq:singular}; its product converges to a positive number.
\end{proposition}

Put
\[
B_k=\frac{V^{k-1}L^{k-1}}{(k-1)!},\qquad
 D_k^0(n)=B_kd_k(n_{\mathcal S}).
\]
So $D_k^0(n)$ is what we obtain if we replace the factor $M_{k,\mathcal S}(n_{\mathcal{S}'})/V$ in the definition of $D_k(n)$ by its dominant term $B_k$. We first explain why it suffices to prove the same asymptotic for every smooth weight with $D_k^0$ in place of $D_k$.

For \(n\asymp x\) with \(n_{\mathcal S}\le E\), Lemma~\ref{lem:density-size} applied with \(A=n_{\mathcal S}\) gives
\[
|D_k(n)-D_k^0(n)|
\ll L^{-1/2}D_k^0(n).
\]
Choose a nonnegative smooth majorant of \(|\Phi|\) with compact support in \((0,\infty)\). Assuming the asserted asymptotic for \(D_k^0\) for every smooth weight, its application to this majorant bounds the sum against \(D_k^0(n)d_l(n+h)\) by \(O(xL^{k+l-2})\). Consequently, the difference restricted to \(n_{\mathcal S}\le E\) is \(O(xL^{k+l-5/2})\). Finally, the complementary range \(n_{\mathcal S}>E\) contributes \(O(xL^C\exp(-cL^{1/10}))\) to the difference between the two sums by \eqref{eq:full-record-tail} and the Cauchy--Schwarz inequality. It therefore remains to establish the asserted asymptotic for \(D_k^0\), which will be our goal in the remainder of this section.

\paragraph{1. Removing large factors.}
Recall the parameters $ \theta=(\log y_0)^{-1}=L^{-2/5}$ and $E=\exp( \sqrt{L})$ used in Section~2.4. We begin by showing that we may discard the large factors from $D_k^{0}$. Write $d_0(1)=1$ and $d_0(n)=0$ for $n>1$.
The standard identity
\begin{equation}
d_k(n_{\mathcal S})
   =\sum_{\substack{d\mid n\\d_{\mathcal S}=d}}d_{k-1}(d)
 \label{eq:positive-proxy-divisors}
\end{equation}
follows by multiplying
$\sum_{e=0}^{\nu}d_{k-1}(p^e)=d_k(p^\nu)$ over primes.
On the other hand, by Rankin's inequality
\[
\sum_{\substack{d\mid n,\ d_{\mathcal S}=d\\d>E}}d_{k-1}(d)
 \leq E^{-\theta}n_{\mathcal S}^{\theta}d_k(n_{\mathcal S}).
\]
By the Cauchy-Schwarz inequality and two applications of Lemma~\ref{lem:record-moments}, we conclude that

$$
B_k\sum_n |\Phi(n/x)|\,d_l(n+h)
 \sum_{\substack{d\mid n\\ d_{\mathcal S}=d\\ d>E}} d_{k-1}(d) \ll
xL^C\exp(-L^{1/10}). $$

\paragraph{2. Evaluating a single progression.}
Fix a retained divisor $d\leq E$ and put $g=(d,h)$ and $q=d/g$.
On writing $n+h=gm$, the condition $d\mid n$ becomes
\begin{equation}
m\equiv h/g\pmod q,\qquad (h/g,q)=1.
 \label{eq:scalar-reduced-progression}
\end{equation}
We are thus interested in evaluating
$$
S_d
=
\sum_{\substack{m\ge1\\m\equiv h/g\;(\mathrm{mod}\ q)}}
\Phi\left(\frac{gm-h}{x}\right)d_l(gm).
$$
To estimate this we notice that the Euler factors give, for $\Re s>1$,
\[
\sum_{m\geq1}\frac{d_l(gm)\chi(m)}{m^s}
 =L(s,\chi)^l
      \prod_{p\mid g}P_{l,v_p(g)}(\chi(p)p^{-s}),
\]
where for $a\geq0$, we define
\[
P_{l,a}(z)=(1-z)^l\sum_{\nu\geq0}d_l(p^{a+\nu})z^\nu = (1-z)^l\sum_{\nu\geq0}\binom{l+a+\nu-1}{l-1}z^\nu.
\]
Notice that $P_{l,a}$ is a polynomial of degree at
most $l-1$, independent of $p$, with
$P_{l,0}=1$ and $P_{l,a}(0)=d_l(p^a)$.
Since $g\mid|h|$, the extra factor in the Dirichlet series and its fixed derivatives are
bounded in terms of $h,l$ on every fixed vertical strip.

We shall apply character orthogonality to
\eqref{eq:scalar-reduced-progression} and Mellin inversion with
\[
\widehat\Phi_{h/x}(s)
     =\int_0^\infty\Phi(u-h/x)u^{s-1}\,du.
\]
This gives
\[
\begin{aligned}
S_d={}&\frac{1}{2\pi i\varphi(q)}
 \sum_{\chi\bmod q}\overline{\chi(h/g)}
 \int_{2-i\infty}^{2+i\infty}
       L(s,\chi)^l\widehat\Phi_{h/x}(s)\\
&\qquad{}\times\left(\frac{x}{g}\right)^s
       \prod_{p\mid g}P_{l,v_p(g)}(\chi(p)p^{-s})\,ds.
\end{aligned}
\]
We shift the contour to $\Re s=1/2$. We use the standard uniform estimate
\[
|L(s,\chi)|\ll q^{1/2}(1+|\Im s|)
\qquad\left(\tfrac12\leq\Re s\leq2,\quad |s-1|\geq\tfrac12\right)
\]
(see \cite[Exercise~44 and (48)]{TaoNotes2}). Together with the rapid
decay of $\widehat\Phi_{h/x}$ uniformly in this strip, this makes the
horizontal integrals tend to zero. Only the principal character
$\chi_0$ modulo $q$ contributes a residue and we obtain
\begin{equation}
\label{Sdterm}
S_d
 =\Res_{s=1}\zeta(s)^lK_{d,h}(s)x^s\widehat\Phi_{h/x}(s)
       +O_{h,l,\Phi}(x^{1/2}q^{l/2}),
\end{equation}
where 
\[
K_{d,h}(s)
=
\frac{g^{-s}}{\varphi(q)}
\prod_{p\mid q}(1-p^{-s})^l
\prod_{p\mid g}
P_{l,v_p(g)}
\left(\chi_0(p)p^{-s}\right).
\]
Notice that for fixed $s$ this is a multiplicative function in $d$. Putting $r=v_p(h)$, its factors are
\begin{equation}
K_{p^e,h}(s)=
 \begin{cases}
 p^{-es}P_{l,e}(p^{-s}),&0\leq e\leq r,\\[2mm]
 \displaystyle
 \frac{p^{-rs}d_l(p^r)(1-p^{-s})^l}{\varphi(p^{e-r})},
                                  & e>r.
 \end{cases}
 \label{eq:scalar-progression-factors}
\end{equation}
Indeed, if $e\leq r$ then $p\nmid q$, while if $e>r$ then
$p\mid q$ and the principal character has value zero at $p$, so we recover $P_{l,v_p(g)}(0)=d_l(p^r)$.

\paragraph{3. Estimating the residue.}
Notice that since $d\leq E=x^{o(1)}$, summing the errors in (\ref{Sdterm}) over all such $d$ with the
coefficients $B_kd_{k-1}(d)$ costs only $x^{1/2+o(1)}$. Thus it only remains to estimate the weighted sum of the residues
\[
R_d:=\Res_{s=1}\zeta(s)^lK_{d,h}(s)x^s\widehat\Phi_{h/x}(s).
\]
We will look at the Taylor expansions of the non-zeta factors and show that the dominant term in the residue comes from the $(l-1)$th derivative of  $x^s=xe^{(s-1)L}$. The fixed derivatives of $\widehat\Phi_{h/x}$ near one are bounded uniformly in $x$. We now show that, for each fixed $j\geq0$,
\[
|K_{d,h}^{(j)}(1)|
          \ll_{h,l,j}\frac{(1+\log d)^j}{d}.
\]
For $p\nmid h$ and $e\geq1$, the local formula
\eqref{eq:scalar-progression-factors} simplifies to
\[
K_{p^e,h}(s)=\frac{p^{-e}}{1-p^{-1}}(1-p^{-s})^l.
\]
Thus $K_{p^e,h}(1)\leq p^{-e}$, and differentiating its logarithm
gives
$|({\log K_{p^e,h}})^{(j)}(1)|\ll_{l,j}(\log p)^j/p$ for $j\geq1$.
For $p\mid h$, put $r=v_p(h)$. There are only finitely many cases
$e\leq r$, while for $e>r$ the factor $K_{p^e,h}(s)$ is $p^{-e}$
times a function independent of $e$. After extracting $p^{-e}$, these primes therefore contribute
only bounded factors and bounded logarithmic
derivatives near one.
Adding the logarithmic derivatives over the primes $p\mid d$ not dividing $h$ and using
\[
\sum_{p\mid d}\frac{(\log p)^j}{p}\leq(\log d)^j
       \qquad(j\geq1),
\]
we obtain $|({\log K_{d,h}})^{(j)}(1)|\ll_{h,l,j}(1+\log d)^j$.
Together with $K_{d,h}(1)\ll_{h,l}d^{-1}$, repeated differentiation
of $K_{d,h}'=K_{d,h}(\log K_{d,h})'$ proves the claimed bound.

Now put $\kappa_{d,h}=K_{d,h}(1)$ and note that
$\widehat\Phi_{h/x}(1)=\int\Phi$ for sufficiently large $x$.
Using that $d \le E$, we see that
\[
R_d=\frac{xL^{l-1}}{(l-1)!}\kappa_{d,h}\int\Phi
       +O_{h,l,\Phi}\!\left(\frac{xL^{l-3/2}}d\right).
\]
The contribution of the error term to the $d$-sum is then controlled by 
\[
x B_k L^{l-3/2}
\sum_{\substack{d_{\mathcal S}=d}}
\frac{d_{k-1}(d)}{d}
\ll_{h,l,\Phi}
xB_k L^{l-3/2}V^{-(k-1)}
\ll_{k,l,h,\Phi}
xL^{k+l-5/2}.
\]
To obtain an Euler product out of the main term of $R_d$ we now need to bring back the large divisors $d > E$ that we discarded, which we can safely do since the bound $|\kappa_{d,h}|\ll_{h,l}d^{-1}$ and Rankin's inequality give
\[
\sum_{\substack{d>E\\d_{\mathcal S}=d}}
 d_{k-1}(d)|\kappa_{d,h}|
 \ll E^{-\theta}
       \prod_{p\in\mathcal S}(1-p^{-1+\theta})^{-(k-1)}
 \ll \exp(-L^{1/10})L^C.
\]
The final bound follows as in Lemma~\ref{lem:record-moments}; for
$k=1$ the tail is zero. We have therefore proved
\[
\sum_n\Phi(n/x)D_k^0(n)d_l(n+h)
 =\frac{xL^{k+l-2}}{(k-1)!(l-1)!}
       \prod_{p\in\mathcal S}c_{p,h}\int\Phi
       +O(xL^{k+l-5/2}),
\]
where
\[
c_{p,h}=(1-p^{-1})^{k-1}
       \sum_{e\geq0}d_{k-1}(p^e)\kappa_{p^e,h}.
\]

\paragraph{4. Identifying the singular series.}
Put $v=1-p^{-1}$ and $r=v_p(h)$, and let
\[
J_{p,e}=\int_{\mathbb Z_p}\mathbf1_{\{p^e\mid u\}}
                         d_l(p^{v_p(u+h)})\,du.
\]
If $e\leq r$, translation by $h$ preserves $p^e\mathbb Z_p$, so
\[
J_{p,e}=v\sum_{\nu\geq e}d_l(p^\nu)p^{-\nu}.
\]
If $e>r$, every $u\in p^e\mathbb Z_p$ has $v_p(u+h)=r$, and hence
$J_{p,e}=p^{-e}d_l(p^r)$.
Comparison with \eqref{eq:scalar-progression-factors} at $s=1$
gives $\kappa_{p^e,h}=v^{l-1}J_{p,e}$ in both cases.
Integrating the nonnegative expansion
\[
d_k(p^{v_p(u)})=\sum_{e\geq0}d_{k-1}(p^e)
                                      \mathbf1_{\{p^e\mid u\}}
\]
term by term therefore gives
\[
c_{p,h}=v^{k+l-2}\int_{\mathbb Z_p}d_k(p^{v_p(u)})
                                d_l(p^{v_p(u+h)})\,du.
\]
These integrals are finite and positive.
For $p\nmid h$, summing the factors
$\kappa_{p^e,h}=p^{-e}v^{l-1}$ for $e\geq1$ yields
\begin{equation}
c_{p,h}=v^{k-1}+v^{l-1}-v^{k+l-2}
       =1+O_{k,l}(p^{-2}).
 \label{eq:unramified-scalar-factor}
\end{equation}
Since $c_{p,h}=v^{k+l-2}I_{p,h}$, this also gives
\begin{equation}\label{eq:local-unramified}
 I_{p,h}=v^{1-k}+v^{1-l}-1
 \qquad(p\nmid h).
\end{equation}
Thus $\prod_pc_{p,h}=\mathfrak S_{k,l}(h)$ converges to a positive
number. Once $\mathcal S$ contains every prime divisor of $h$,
the omitted factors cost $O(\sum_{p>y_0}p^{-2})=O(y_0^{-1})$,
which is smaller than $O(L^{-1/2})$.
This proves \eqref{eq:scalar-proxy-asymptotic} with $D_k^0$ in
place of $D_k$.

\section{Reconstruction from banks with one prime factor}
\label{sec:reconstruction}

By Proposition~\ref{prop:scalar-proxy}, it remains to bound the
correlation of $\xi_k(n)$ with $d_l(n+h)$. In this section,
we will show that the problem can be reduced to correlations
in which both $n$ and $n+h$ have exactly one distinct prime
divisor in many of
the banks introduced in \eqref{eq:banks}. After discarding negligible contributions from higher prime
powers, each selected bank then contributes the fixed value
$kl$ to $\xi_k(n)d_l(n+h)$. This allows us to factor out
these local contributions and thus transfer the problem to the
$\alpha_{\mathcal P}(n)\beta_{\mathcal P}(n+h)$ part of the correlation
involving dilation-invariant factors.
This will then be handled by the graph machinery we shall develop later, which will also benefit substantially from being able to fix the number of distinct prime divisors coming from each bank.

For an active union $\mathcal P$ of banks, put
\begin{equation}\label{eq:singleton-weight}
 \Omega_{\mathcal P}(n)=
 \prod_{g\in\mathcal G_{\mathcal P}}
          \frac{\one_{\{\omega_g(n)=1\}}}{V_g}.
\end{equation}
Notice that $\Omega_{\mathcal P}(n)\le L^{O(1)}$. The goal of this section is to show that Theorem~\ref{thm:main} can be deduced from the following result, whose proof is deferred to Sections~\ref{sec:analytic-services}--\ref{sec:graph}.

\begin{proposition}[Cancellation with exactly one prime in each bank]
\label{prop:centered-marked-correlation}
Fix $k,l\ge2$, $h\ne0$, and $c>0$.  For every $A>0$ and $\Psi\in C_c^\infty((0,\infty))$, we have
\begin{equation}\label{eq:bridge-producer}
 \sum_{\substack{n\ge1\\n+h\ge1}}\frac{\Psi(n/x)}n
 \alpha_{\mathcal P}(n)\beta_{\mathcal P}(n+h)
 \Omega_{\mathcal P}(n)\Omega_{\mathcal P}(n+h)
       \ll_{A,k,l,h,c,\Psi}L^{-A}
\end{equation}
 uniformly over all active sets \(\mathcal P\) with $|\mathcal G_{\mathcal P}\cap\mathcal G_s| , |\mathcal G_{\mathcal P}\cap\mathcal G_b|\ge cJ$.
\end{proposition}

\subsection{A uniform upper bound for two affine forms}
We will need to upper bound the contribution to the correlation produced by integers with an unusual number of prime divisors in the banks. For this, an important role will be played by the following fixed-resultant consequence of the correlation
bound \cite[Lemma~8.3]{MRSTT},
itself deduced from Henriot's theorem \cite{Henriot, HenriotErr}.

\begin{corollary}\label{cor:henriot-record-weights}
Fix $j_1,j_2\ge1$ and a nonzero integer $h$, and write $L=\log x$.
Let $R=x^{1+o(1)}$, and let $Q_i(m)=a_im+b_i$ be primitive affine
forms with $a_i>0$, coefficients $\exp(O(L^{1/2}))$, and nonzero
resultant $\Delta=a_1b_2-a_2b_1$ dividing $h$.
For arbitrary sets of primes $\mathcal U_i,\mathcal V_i$, put
\[
 f_i(n)=d_{j_i}(n_{\mathcal V_i})\one_{(n,\mathcal U_i)=1},\qquad
 I_p=\int_{\mathbb Z_p}\prod_{i=1}^2
              f_i\!\left(p^{v_p(Q_i(u))}\right)\,du.
\]
Then, uniformly in the above choices of $Q_i$, $\mathcal U_i$ and $\mathcal V_i$, we have
\begin{equation}\label{eq:henriot-record-upper}
 \sum_{R<m\le 2R}f_1(Q_1(m))f_2(Q_2(m))
       \ll_{j_1,j_2,h}R\prod_{p\le R}I_p.
\end{equation}
\end{corollary}

\begin{proof}
We may clearly replace $f_i$ with $f_i (n) \one_{(n,a_i)=1}$ without affecting either side of (\ref{eq:henriot-record-upper}). We also translate the variable $m$ by $\exp(O(L^{1/2}))=o(R)$ so that the constant
terms of both $Q_i$ become positive. The translated interval is contained in
$[\sqrt X,X]$ for some $X\asymp R$, and evidently all new coefficients are at most
$2X$ for sufficiently large $x$. We may therefore apply
\cite[Lemma~8.3]{MRSTT}, using that both $f_i$
are submultiplicative and satisfy the growth hypotheses of its statement,
uniformly in both prime sets. Indeed, $(ab)_{\mathcal V_i}=a_{\mathcal V_i}b_{\mathcal V_i}$,
$d_j(ab)\le d_j(a)d_j(b)$, and
\[
 f_i(p^\nu)\le d_{j_i}(p^\nu)\le j_i^\nu,\qquad
 f_i(n)\le d_{j_i}(n)\ll_{j_i,\varepsilon}n^\varepsilon.
\]
The discriminant parameter $D$ in the statement
is $D=\Delta^2\mid h^2$, so the corresponding $\Delta_D$ factor there is
$O_{j_1,j_2,h}(1)$ by \cite[(8.12)]{MRSTT}.

The bound in \cite[Lemma~8.3]{MRSTT} is in terms of two harmonic sums over the nonnegative $f_i$. We bound them, up to a factor depending only on $j_1,j_2,h$, by their Euler products over all primes $p \le \sqrt{X}$ that do not divide $h$. For $p\nmid2h$, write $v=1-p^{-1}$ and
$G_i=\sum_{\nu\ge0}f_i(p^\nu)p^{-\nu}\in\{1,v^{-1},v^{-j_i}\}$.
If $p$ divides $a_i$, the condition $(n,a_i)=1$ in $f_i$ gives $G_i=1$
and the resulting local upper factor in the statement equals $vG_{3-i}=I_p$.
Otherwise $a_1, a_2$ are units and the roots of $Q_1, Q_2$ are distinct mod $p$, so the
upper factor in the inequality of \cite[Lemma~8.3]{MRSTT} and the Haar mean are respectively
\[
 (1-2/p)G_1G_2,\qquad I_p=v(G_1+G_2)-1.
\]
Since $G_i=1+O_{j_i}(p^{-1})$ their ratio is
$1+O_{j_1,j_2}(p^{-2})$, leading to a bounded product.
For the $O_h(1)$ primes dividing $h$ it suffices to obtain a uniform lower bound for $I_p$. For those distinct from $2$, this follows from
$I_p\ge1-2/p>0$. At $2$, either no parity class is permitted and the
sum vanishes, or $I_2\ge1/2$; its other local factors are bounded.
Finally the cutoff $\sqrt X\asymp\sqrt R$ can be extended to $R$
at bounded cost by Mertens' theorem and the general lower bound $I_p\ge1-2/p$.
\end{proof}

\subsection{An enveloping measure on the records}

Recall that
$\mathcal C=\bigcup_{g\in\mathcal G_s\cup\mathcal G_b}P_g$
is the union of all the small and big banks.
For sufficiently large $x$, no prime in $\mathcal C$
divides $h$, so the records $n_{\mathcal C}$ and
$(n+h)_{\mathcal C}$ are coprime.

Fixing the values
\[
 n_{\mathcal C}=A,\qquad (n+h)_{\mathcal C}=B,
\]
the Chinese remainder theorem allows us to express $n/A$ and
$(n+h)/B$ as two primitive affine forms in a common
integer variable. This brings the upper bound of
Corollary~\ref{cor:henriot-record-weights} into play
for the correlation restricted to that record pair.
Our goal in this subsection is to show that the
resulting bounds over all pairs $(A,B)$ can be organized into a
single probability law that controls arbitrary
weighted collections of such records.

To make this precise, fix $p\in\mathcal C$ and
let $u$ be distributed according to normalized
Haar measure on $\mathbb Z_p$.
Since $p\nmid h$, the valuations $v_p(u)$ and
$v_p(u+h)$ cannot both be positive. Writing
$v=1-p^{-1}$, the pair $(0,0)$ has probability
$1-2/p$, while $(j,0)$ and $(0,j)$ each have
probability $v/p^j$ for $j\ge1$.

The local divisor weight
$d_k(p^{v_p(u)})d_l(p^{v_p(u+h)})$
has mean $I_{p,h}$, whose value was explicitly computed in
\eqref{eq:local-unramified}.
We now weight these valuation probabilities
by the corresponding divisor factors and divide
by $I_{p,h}$. This defines a probability measure
$\mu_p$ on pairs of valuations, given by:
\begin{align}
 \mu_p(0,0)&=(1-2/p)/I_{p,h},
                  \label{eq:record-law}\\
 \mu_p(j,0)&=v\,d_k(p^j)/(p^jI_{p,h})
                  \quad(j\ge1),\notag\\
 \mu_p(0,j)&=v\,d_l(p^j)/(p^jI_{p,h})
                  \quad(j\ge1).\notag
\end{align}
Pairs with both valuations positive receive mass
zero. Let $\mu_{\mathcal C}$ be the independent
product of these local measures over
$p\in\mathcal C$. It is a probability law on
pairs of $\mathcal{C}$-records $(A,B)$. We will prove the following result.

\begin{lemma}\label{lem:positive-record-law}
Let $\mathcal C$ be as in \eqref{eq:fullrecord}. For every function $0\le F(A,B)\le1$ the sums
\[
 \sum_{n\asymp x}d_k(n)d_l(n+h)F(n_{\mathcal C},(n+h)_{\mathcal C}),
 \qquad
 \sum_{n\asymp x}D_k(n)d_l(n+h)F(n_{\mathcal C},(n+h)_{\mathcal C})
\]
are
\begin{equation}\label{eq:record-event-upper}
 \ll_{k,l,h} xL^{k+l-2}\E_{\mu_{\mathcal C}}F
       +xL^C e^{-cL^{1/10}},
\end{equation}
with the implied constant independent of $F$ and of the number of banks.
\end{lemma}

\begin{proof}
It suffices to work on a dyadic interval $x<n\le2x$. Also, by Lemma~\ref{lem:record-moments} and the Cauchy-Schwarz inequality, it is enough to bound the sums restricted to those $n$ with $n_{\mathcal{S}}, (n+h)_{\mathcal{S}}  \le E=\exp(\sqrt L)$. We handle the $d_k$ and $D_k$ cases separately.

\paragraph{1. The actual divisor weight.}
Notice that once $x$ is sufficiently large with respect to $h$ we need only consider the values of $F(A,B)$ at coprime integers $A, B \le E$ that are both supported on $\mathcal{C}$ . Let us consider one such pair and estimate the contribution from those integers $n$ with $n_{\mathcal C}=A$ and $(n+h)_{\mathcal C}=B$. By positivity it suffices to bound
\[
 S(A,B):=\sum_{\substack{x<n\le2x\\
              n_{\mathcal C}=A,\ (n+h)_{\mathcal C}=B}}
                         d_k(n)d_l(n+h).
\]
By the Chinese remainder theorem, the divisibilities $A\mid n$ and
$B\mid n+h$ specify one residue class modulo $AB$. Choose its
representative $-AB<r\le0$ and write $n=ABm+r$. The remaining factors are then given by the affine forms
\[
 Q_1(m)=\frac nA=Bm+r/A,\qquad
 Q_2(m)=\frac{n+h}{B}=Am+(r+h)/B.
\]
The interval $x<n\le2x$ becomes $R<m\le R+Y$, where
\[
 R=\frac{x-r}{AB}=x^{1-o(1)},\qquad
 Y=\frac{x}{AB}\asymp R,\qquad Y\le R.
\]
Precisely, the sum becomes
\[
 S(A,B)=d_k(A)d_l(B)
 \sum_{\substack{R<m\le R+Y\\
                 (Q_1(m)Q_2(m),\mathcal C)=1}}
                          d_k(Q_1(m))d_l(Q_2(m)).
\]
This is the type of sum controlled by
Corollary~\ref{cor:henriot-record-weights}. The forms are primitive:
a common prime divisor of the two coefficients of either form would
divide both $AB$ and $h$, which is impossible. Their resultant is $h$
and their coefficients are $\exp(O(L^{1/2}))$ because $A,B\le E$
and $|r|<AB$. Since $(R,R+Y]\subseteq(R,2R]$ and $R\asymp x/(AB)$,
this corollary with $j_1=k$, $j_2=l$ and
$\mathcal U_1=\mathcal U_2=\mathcal C$ gives
\[
 S(A,B)\ll_{k,l,h}\frac{x\,d_k(A)d_l(B)}{AB}
                              \prod_{p\le R}I_p,
\]
where, in this application,
\[
I_p
=
\int_{\mathbb Z_p}
d_k\!\left(p^{v_p(Q_1(u))}\right)
d_l\!\left(p^{v_p(Q_2(u))}\right)
\mathbf 1_{\{(p\in\mathcal C)\Rightarrow v_p(Q_1(u))=v_p(Q_2(u))=0\}}
\,du.
\]
We now identify these local factors more precisely. If $p \in \mathcal C$ does not divide $AB$ then $p$ divides neither the slopes of the $Q_i$ nor their resultant $h$. Thus the roots of the $Q_i$ are unique and distinct mod $p$, so $I_p=1-2/p$. If $p \in \mathcal{C}$ does divide $AB$, one form
is a constant unit modulo $p$ and the other has one forbidden root,
so $I_p=1-1/p$. Outside $\mathcal C$, the substitution $n=ABm+r$
preserves Haar measure and division by $A,B$ does not change
valuations; hence $I_p=I_{p,h}$. In particular, we have
\[
\prod_{\substack{p\le R\\ p\notin\mathcal C}} I_{p,h}
=
\left\{\prod_{p\le R}(1-p^{-1})^{-1}\right\}^{k+l-2}
\left(\prod_{p\le R}c_{p,h}\right)
\left(\prod_{p\in\mathcal C}I_{p,h}\right)^{-1}
\ll_{k,l,h}
L^{k+l-2}
\left(\prod_{p\in\mathcal C}I_{p,h}\right)^{-1},
\]
where $c_{p,h}$ are the local factors introduced in the proof of
Proposition~\ref{prop:scalar-proxy}, whose product was shown to converge
to a positive limit in part~4 of that proof. We thus conclude that
\[
\begin{aligned}
S(A,B)
&\ll_{k,l,h}
xL^{k+l-2}\,
\frac{d_k(A)d_l(B)}{AB}
\prod_{p\in\mathcal C}\frac{I_p}{I_{p,h}}\\
&=
xL^{k+l-2}\mu_{\mathcal C}(A,B),
\end{aligned}
\]
where the last equality is immediate from the definitions. Multiplying by $F(A,B)$ and summing over all such pairs $A,B \le E$ proves
the result for $d_k$.

\paragraph{2. The density model case.}
The argument is essentially the same as the previous one with only slight modifications. Let $T(A,B)$ denote
the sum $S(A,B)$ with $d_k(n)$ replaced by $D_k(n)$.
By (\ref{Mkbound}) and \eqref{eq:model-def}, uniformly for $x<n\le2x$,
\[
 D_k(n)\ll_k V^{k-1}L^{k-1}d_k(n_{\mathcal S}).
\]
For $n=AQ_1(m)$ in this sum,
$n_{\mathcal S}=A(Q_1(m))_{\mathcal S\setminus\mathcal C}$.
Thus the same CRT argument leaves the residual weights
\[
 g_1(t)=d_k(t_{\mathcal S\setminus\mathcal C})
                  \one_{(t,\mathcal C)=1},\qquad
 g_2(t)=d_l(t)\one_{(t,\mathcal C)=1}.
\]
Corollary~\ref{cor:henriot-record-weights} applies to these weights
and the same primitive forms $Q_1,Q_2$. For primes $p \in \mathcal C$ their
local means are the factors $I_p$ calculated in the first part.
At $p\in\mathcal S\setminus\mathcal C$, both divisor weights are
present and the local mean is $I_{p,h}$, since $A,B$ are $p$-adic
units. Outside $\mathcal S$, the first weight is identically one
and the second has mean $(1-p^{-1})^{1-l}$. We therefore obtain
\[
 \begin{aligned}
 T(A,B)&\ll_{k,l,h}
 \frac{xV^{k-1}L^{k-1}d_k(A)d_l(B)}{AB}
 \left(\prod_{p\in\mathcal C}I_p\right)
 \left(\prod_{p\in\mathcal S\setminus\mathcal C}I_{p,h}\right)\\
 &\quad{}\times\prod_{y_0<p\le R}(1-p^{-1})^{1-l}.
 \end{aligned}
\]
By Mertens' estimate the last product is $\ll_l L^{l-1}V^{l-1}$.
Combining this with the model prefactor and using
\[
 V^{k+l-2}\prod_{p\in\mathcal S\setminus\mathcal C}I_{p,h}
 =\left(\prod_{p\in\mathcal S}c_{p,h}\right)
       \left(\prod_{p\in\mathcal C}I_{p,h}\right)^{-1}
 \ll_{k,l,h}\left(\prod_{p\in\mathcal C}I_{p,h}\right)^{-1},
\]
we conclude, just as in the first part, that
\[
 T(A,B)\ll_{k,l,h}xL^{k+l-2}
 \frac{d_k(A)d_l(B)}{AB}
 \prod_{p\in\mathcal C}\frac{I_p}{I_{p,h}}
 =xL^{k+l-2}\mu_{\mathcal C}(A,B).
\]
Multiplying by $F(A,B)$ and summing over the same candidate pairs
proves the model assertion.

\end{proof}

\subsection{Banks with a single factor}
\label{subsec:good-records}

For a typical integer $n$, the expected number of distinct
prime divisors from a bank $P_g$ is approximately
$V_g\asymp1$. In particular, we expect with positive
probability that both $n$ and $n+h$ have exactly one prime
divisor in $P_g$. Since we have $\asymp\log\log x$
banks of each size, we furthermore expect most integers to have this
property in a positive proportion of these banks. However, it is not enough to show that the exceptions
are rare: we must also control their contribution to the
divisor-weighted correlation. In this subsection we show
how this can be guaranteed thanks to the enveloping measure bound obtained
in Lemma~\ref{lem:positive-record-law}. Put
\[
 I_g(n)=\one_{\{\omega_g(n)=\omega_g(n+h)=1\}}.
\]
Fix $k,l\ge2$ and $h\ne0$, and let
\[
 \beta:=\frac{kl}{8}e^{-2(k+l)},\qquad \eta_0:=\frac{\beta c_*}{8},
\]
with $c_*$ as in \eqref{eq:banks}. Set
\[
 T_s=\lceil\beta|\mathcal G_s|/2\rceil,
 \qquad T_b=\lceil\beta|\mathcal G_b|/2\rceil.
\]
Finally, write $G(n)$ for the indicator function of those $n$ for which at least $T_s$ small banks and
$T_b$ big banks satisfy $I_g(n)=1$. The following consequence of Lemma~\ref{lem:positive-record-law} allows us to restrict attention to such integers $n$.

\begin{corollary}\label{cor:good-records}
With $G$ and $\eta_0$ as above, for every $\Phi\in C_c^\infty((0,\infty))$,
\begin{equation}\label{eq:singleton-bad-arithmetic}
 \sum_n|\Phi(n/x)\xi_k(n)|d_l(n+h)(1-G(n))
 \ll_{k,l,h,\Phi}xL^{k+l-2-\eta_0}+xL^Ce^{-cL^{1/10}}.
\end{equation}
Here $C,c>0$ depend only on $k,l,h,\Phi$.
\end{corollary}

\begin{proof}
Notice that since the banks are pairwise disjoint, the events $I_g(n)$ are independent under the measure
$\mu_{\mathcal C}$. To compute their probabilities, put
\[
 r_p=\mu_p(0,0),\qquad
 a_p=\frac1{r_p}\sum_{j\ge1}\mu_p(j,0),\qquad
 b_p=\frac1{r_p}\sum_{j\ge1}\mu_p(0,j),
\]
with $\mu_p$ as in \eqref{eq:record-law}. Notice that $a_p, b_p$ are normalized by $r_p$. We then have
$r_p=1-(k+l)/p+O_{k,l}(p^{-2})$,
$a_p=k/p+O_{k,l}(p^{-2})$ and $b_p=l/p+O_{k,l}(p^{-2})$.
Summing over the two distinct primes that divide the endpoints gives
\begin{equation}\label{eq:singleton-probability}
 \begin{aligned}
 \mathbb E_{\mu_{\mathcal C}}I_g
 &=\left(\prod_{p\in P_g}r_p\right)
          \sum_{\substack{p,q\in P_g\\p\ne q}}a_pb_q\\
 &=klV_g^2e^{-(k+l)V_g}
        +O_{k,l}\!\left(\sum_{p\in P_g}p^{-2}\right).
 \end{aligned}
\end{equation}
The error tends to zero uniformly over the banks. Using
$1/2\le V_g\le2$, we obtain $ \mathbb E_{\mu_{\mathcal C}}I_g\ge\beta$ for sufficiently large $x$. For independent Bernoulli variables $I_1,\ldots,I_N$ having means
at least $\beta$, the multiplicative Chernoff bound gives
\[
 \mathbb P\!\left(\sum_{j=1}^N I_j<\lceil\beta N/2\rceil\right)
 \le e^{-\beta N/8}.
\]
Thus, $ \mathbb E_{\mu_{\mathcal C}}(1-G)\ll L^{-\eta_0}$.
Since $|\xi_k|\le d_k+D_k$, the result follows by applying
Lemma~\ref{lem:positive-record-law} with $F=1-G$ and bounding
$|\Phi|$ on its support.
\end{proof}

\subsection{Deduction of the main theorem}
\label{sec:completion}

We now deduce Theorem~\ref{thm:main} from Proposition~\ref{prop:centered-marked-correlation} and
Corollary~\ref{cor:good-records}, together with the density model
asymptotic of Proposition~\ref{prop:scalar-proxy}. If $\min\{k,l\}=1$, then $\Sg_{k,l}(h)=1$ and the assertions of
Theorem~\ref{thm:main} follow from the classical divisor mean by
partial summation and a fixed translation. When $k=l=1$, integer
counting suffices. We may therefore assume $k,l\ge2$.

Recall from \eqref{eq:model-def} that $\xi_k=d_k-D_k$. We will use the notation introduced in Section~\ref{subsec:good-records}. 
Corollary~\ref{cor:good-records}
already bounds the contribution with $G(n)=0$, so it remains to
estimate the correlation with $G(n)$ inserted. Recall that $I_g(n)=1$
means that $P_g$ contains exactly one prime divisor for each of
$n$ and $n+h$ and $G(n)=1$ when this happens in at least $T_s$
small banks and $T_b$ big banks. We will express $G(n)$ as a sum
of products of the $I_g(n)$, to which
Proposition~\ref{prop:centered-marked-correlation} applies.

Put $\mathcal G=\mathcal G_s\sqcup\mathcal G_b$ and $N=|\mathcal G|$.
Separating according to the exact set $\mathcal A$ of banks for
which $I_g(n)=1$ and then expanding the factors $1-I_g(n)$ gives
\begin{equation}\label{eq:singleton-boolean}
 G(n)=
 \sum_{\substack{\mathcal A\subseteq\mathcal G\\
       |\mathcal A\cap\mathcal G_s|\ge T_s\\
       |\mathcal A\cap\mathcal G_b|\ge T_b}}
 \prod_{g\in\mathcal A}I_g(n)\prod_{g\notin\mathcal A}(1-I_g(n))
 =\sum_{\mathcal U\subseteq\mathcal G}c_{\mathcal U}
                                      \prod_{g\in\mathcal U}I_g(n).
\end{equation}
For fixed $\mathcal A$, the expansion produces
$2^{N-|\mathcal A|}$ terms. Allowing all subsets $\mathcal A$
bounds their total number by
$\sum_{\mathcal A\subseteq\mathcal G}2^{N-|\mathcal A|}=3^N$.
Moreover, every resulting $\mathcal U$ contains its original
$\mathcal A$, so it still has enough banks of each size. Thus
\begin{equation}\label{eq:singleton-boolean-cost}
 \sum_{\mathcal U}|c_{\mathcal U}|\le3^N=L^{O(1)},
 \qquad
 c_{\mathcal U}\ne0\Longrightarrow
 \begin{cases}
 |\mathcal U\cap\mathcal G_s|\ge T_s,\\
 |\mathcal U\cap\mathcal G_b|\ge T_b.
 \end{cases}
\end{equation}
By the definitions of $T_s,T_b$ and \eqref{eq:banks}, each such
$\mathcal U$ contains at least $(\beta c_*/2)J$ banks of each size,
as required in Proposition~\ref{prop:centered-marked-correlation}.

We now identify the correlation attached to one of these products.
Fix such a $\mathcal U$ and put $\mathcal P=\bigcup_{g\in\mathcal U}P_g$.
The conditions $I_g(n)=1$ prescribe the number of distinct prime
divisors, but still allow higher powers. Let us now discard their contribution. By \eqref{eq:core-factor}
and \eqref{eq:singleton-weight} we have
$|\alpha_{\mathcal P}|\le d_k+D_k$, $\beta_{\mathcal P}\le d_l$
and $\Omega_{\mathcal P}\le L^{O(1)}$. Using also
$D_k\ll_k L^{k-1}d_k$, Cauchy--Schwarz and the standard fourth
divisor moment bounds show that on a fixed dyadic interval the contribution
of higher powers is at most
\[
 \left(\sum_{p\in\mathcal C}O(x/p^2)\right)^{1/2}
 \left(xL^C\right)^{1/2}
 \ll xL^Ce^{-cL^{1/8}},
\]
uniformly in $\mathcal U$. Summing over $\mathcal U$ preserves
negligibility by \eqref{eq:singleton-boolean-cost}.

Outside this exceptional set, $\prod_{g\in\mathcal U}I_g(n)=1$ means that both
$n$ and $n+h$ have exactly one prime factor from each
selected bank, counting multiplicity. The primes in $\mathcal P$ therefore contribute the factors
$k^{|\mathcal U|}$ and $l^{|\mathcal U|}$ to the two divisor weights. Thus, outside of the integers we have just discarded, it follows from \eqref{eq:core-factor} and the definition \eqref{eq:singleton-weight} of
$\Omega_{\mathcal P}$ that
\begin{equation}\label{eq:singleton-reconstruct-identity}
 \xi_k(n)d_l(n+h)\prod_{g\in\mathcal U}I_g(n)
 =\left(\prod_{g\in\mathcal U}klV_g^2\right)
 \alpha_{\mathcal P}(n)\beta_{\mathcal P}(n+h)
 \Omega_{\mathcal P}(n)\Omega_{\mathcal P}(n+h).
\end{equation}
We can now apply Proposition~\ref{prop:centered-marked-correlation}
with $c=\beta c_*/2$ and $\Psi(u)=u\Phi(u)$. Since
$\Psi(n/x)/n=\Phi(n/x)/x$, this gives an arbitrary fixed
logarithmic saving for each ordinary correlation above.
Using $V_g\le2$, we see that the sum of the $\prod_{g\in\mathcal U}klV_g^2$ factors contributes at most $\sum_{\mathcal U} |c_{\mathcal U}| (4kl)^N =L^{O_{k,l}(1)}$ by (\ref{eq:singleton-boolean-cost}). Multiplying \eqref{eq:singleton-boolean} by
$\Phi(n/x)\xi_k(n)d_l(n+h)$, summing over $n$, and choosing the
saving in the proposition large enough to absorb this cost gives
\[
 \sum_n\Phi(n/x)\xi_k(n)d_l(n+h)G(n)\ll_A xL^{-A}
 \qquad\text{for every fixed }A>0.
\]
Combining this with Corollary~\ref{cor:good-records} we conclude that
\[
 \sum_n\Phi(n/x)\xi_k(n)d_l(n+h)
                 \ll xL^{k+l-2-\eta_0}.
\]
Adding the density model asymptotic of
Proposition~\ref{prop:scalar-proxy} then proves \eqref{eq:smooth-main}
with any fixed $0<\eta<\min(\eta_0,1/2)$.

Finally, positivity gives the usual passage to a sharp cutoff.
Lemma~\ref{lem:positive-record-law} with $F=1$ and dyadic summation
give $\sum_{n\le y}d_k(n)d_l(n+h)\ll_{k,l,h}
y\{\log(2y)\}^{k+l-2}$. Approximate
$\one_{[\varepsilon,1]}$ from above and below by fixed nonnegative
smooth functions with integrals $1+O(\varepsilon)$.
The segment $n\le\varepsilon x$ contributes
$O_{k,l,h}(\varepsilon xL^{k+l-2})$; applying
\eqref{eq:smooth-main} and letting first $x\to\infty$ and then
$\varepsilon\to0$ proves \eqref{eq:sharp-main}.

\section{Centered major-arc estimates}
\label{sec:analytic-services}

To establish Proposition~\ref{prop:centered-marked-correlation},
we will need to control the major-arc behavior of
$\alpha_{\mathcal P}$ through expressions related to its
Dirichlet polynomials $\sum_n\alpha_{\mathcal P}(n)n^{it}$.
As usual, we will control large values of $|t|$
through cancellation in the corresponding prime sums, while
small values of $|t|$ will be handled using the mean-value
behavior of $\alpha_{\mathcal P}$, whose main term vanishes
by construction. The purpose of this section is to establish
the estimates underlying both arguments.

\subsection{The smooth prime sums}
\label{sec:zero-free-prime-services}

We need cancellation in prime sums twisted by $p^{it}$, for an appropriate range of
$|t|$. We give below the short contour calculation required. We will take advantage of the zero-free region of \cite[Theorem~1.1]{OpenAIZeroFree}, although a weaker estimate would suffice. Write
$\delta_\chi=\one_{\{\chi=\chi_0\}}$, where $\chi_0$ is the
principal character for the modulus in question.

\begin{lemma}[Smooth twisted prime sums]
\label{lem:zf-smooth-prime}
Fix $a,B,C>0$, put $L=\log x$,
and suppose
\[
 P\ge\exp(L^a),\qquad q\le L^B,\qquad |t|\le x^C.
\]
Let $w$ be smooth, supported in $[1/2,3]$, with
$(u\partial_u)^jw(u)\ll_j L^{C_j}$ for every fixed $j$. There is a fixed $D$ such that,
uniformly for characters modulo $q$,
\begin{equation}
 \frac1P\sum_p\chi(p)w(p/P)p^{it}
 =\delta_\chi P^{it}\int_0^\infty
       \frac{w(u)u^{it}}{\log(Pu)}\,du
       +O(P^{-1/16}L^D).
 \label{eq:zf-smooth-prime}
\end{equation}
For every fixed $j\ge0$, the integral on the right is
$O_j(L^{D_j}(1+|t|)^{-j}/\log P)$, with a fixed $D_j$.
\end{lemma}

\begin{proof}
The standard local logarithmic-derivative formulas
\cite[Proposition~19 and Exercise~45(ii)]{TaoNotes2} together with the zero-free region of \cite[Theorem~1.1]{OpenAIZeroFree} give
\begin{equation}
 \frac{L'}L(s,\chi)+\frac{\delta_\chi}{s-1}
       \ll\log\bigl(q(|v|+3)\bigr),
 \qquad s=\sigma+iv,\quad 15/16\le\sigma\le2.
 \label{eq:zf-log-derivative}
\end{equation}
Indeed, take the radius in those formulas smaller than $1/16$,
so that no zero contributes. The principal-character Euler factors
\cite[(46)]{TaoNotes2} add $O(\log q)$.

To pass from von Mangoldt weights to primes, set
\[
 f_P(u)=\frac{w(u)}{\log(Pu)},\qquad
 \widehat f_P(s)=\int_0^\infty f_P(u)u^{s-1}\,du.
\]
Logarithmic integration by parts gives, uniformly in the strip,
\[
 |\widehat f_P(\sigma+iz)|
       \ll_j\frac{L^{D_j}}{\log P}(1+|z|)^{-j}.
\]
Mellin inversion therefore expresses the von Mangoldt sum as
\[
 \frac1P\sum_{n\ge2}\Lambda(n)\chi(n)
             \frac{w(n/P)}{\log n}n^{it}
 =\frac1{2\pi i}\int_{(2)}
       -\frac{L'}L(s,\chi)P^{s-1+it}\widehat f_P(s+it)\,ds.
\]
Shift to $\Re s=15/16$. The only possible pole is the principal
pole at $s=1$, whose residue is
$\delta_\chi P^{it}\widehat f_P(1+it)$, giving the stated main
term. By \eqref{eq:zf-log-derivative}, the new integral is at most
\[
 P^{-1/16}\frac{L^{D_2}}{\log P}
 \int_{\mathbb R}
       \frac{1+\log(q(|z-t|+3))}{(1+|z|)^2}\,dz
       \ll P^{-1/16}L^D.
\]
Here $z=\Im s+t$ and the height costs only
$\log(q(|t|+3))\ll L$. The same bounds make the horizontal
integrals vanish.

At $n=p^r$, the factor $\Lambda(n)/\log n$ equals $1/r$.
The terms with $r\ge2$ contribute
$O(L^{C_0}P^{-1/2}\log(2P))$, which is absorbed by the error.
Finally, logarithmic integration by parts in
$\widehat f_P(1+it)$ gives the asserted decay of the principal
integral.
\end{proof}

\subsection{The exact rough divisor sum}
\label{sec:exact-rough-divisor}

We next establish the cancellation needed at low Mellin
frequencies. Writing $n=Av$, where $A=n_{\mathcal S}$
and $(v,\mathcal S)=1$, the definition of our density
model reduces the problem to comparing the mean values
of $d_j(v)$ and $M_{j,\mathcal S}(v)/V$.
Since their corresponding principal residues agree exactly,
this will yield the needed estimate below.

Retain $L=\log x$, $y_0=\exp(L^{2/5})$, $\mathcal S=\{p:p\le y_0\}$,
$V=\prod_{p\in\mathcal S}(1-p^{-1})$, and $E=\exp(\sqrt L)$
from Section~\ref{sec:records}. Put
\[
 \theta=(\log y_0)^{-1}=L^{-2/5},\qquad
 F_{a,\mathcal S}(s)=\prod_{p\in\mathcal S}(1-p^{-s})^a
 \quad(a\ge1).
\]
We will establish the following general result that we shall require later.

\begin{lemma}\label{lem:combined-mellin}
Fix $C>0$. Let $R=x^{1+o(1)}$ and let $q\le L^C$ be $\mathcal{S}$-smooth. Fix integers $j,K\ge1$ and
a real $D\ge0$. For every $\mathcal S$-smooth $A$, let $b_A$ be a $q$-periodic function with $|b_A|\le1$ and let
$\psi_A$ have support in one fixed compact subinterval of
$(0,\infty)$ and satisfy
$\sup_A\|(u\partial_u)^m\psi_A\|_\infty\ll_m L^{C_m}$
for every $m$. Let
\begin{equation}
 S_A(\tau)=\sum_{\substack{v\ge1\\(v,\mathcal S)=1}}
      b_A(v)\psi_A(Av/R)(Av)^{i\tau}
       \left\{d_j(v)-\frac{M_{j,\mathcal S}(v)}V\right\}.
 \label{eq:centered-model}
\end{equation}
Then, there exist constants $c,C'>0$, depending only on the fixed data
above, such that, uniformly over $|\tau|\le L^D$,
\begin{equation}
 \frac1R\sum_{A = A_{\mathcal S}}d_K(A)|S_A(\tau)|
       \ll L^{C'}\exp(-cL^{1/10}),
 \label{eq:centered-mellin-bound}
\end{equation}
 where the sum is over all $\mathcal S$-smooth choices of $A$.
\end{lemma}

\begin{proof}
The two sums will have the same principal residue, while the contribution of the
nonprincipal characters will be shown to be negligible. We first truncate
$A$, then make this comparison uniformly in $A$ and in the
periodic tests.

\paragraph{1. Removing large values of $A$.}
Let $\mathcal R(\tau)$ denote the left side of
\eqref{eq:centered-mellin-bound} and let $\mathcal R_E(\tau)$ denote its restriction to $A\le E$.
Let us first discard the contribution of $S_A$ for $A > E$.
 For $B=\max(K,j)$, the absolute
coefficient of each term inside $S_A$ contributes at most
$R^{-1} L^{O(1)}d_B(n)\one_{\{n_{\mathcal S}>E\}}$ to $\mathcal{R}(\tau)$.
Indeed, $(A,v)=1$ gives $d_K(A)d_j(v)\le d_B(n)$, $A=n_{\mathcal S}$ and
$M_{j,\mathcal S}(v)/V\ll L^{O(1)}$ on the support, by \eqref{Mkbound}.
Since $n\asymp R=x^{1+o(1)}$, Lemma~\ref{lem:record-moments}
therefore gives
\begin{equation}
 0\le\mathcal R(\tau)-\mathcal R_E(\tau)
       \ll L^{O(1)}\exp(-L^{1/10})
       \qquad(\tau\in\mathbb R).
 \label{eq:centered-record-tail}
\end{equation}

\paragraph{2. Fixing $A$ and a character.}
The low-frequency restriction allows us to absorb the Mellin
phase into the smooth cutoff. Indeed, $
 (Av)^{i\tau}=R^{i\tau}(Av/R)^{i\tau}$ and  $\psi_A(u)u^{i\tau}$ still has logarithmic derivatives
bounded by fixed powers of $L$ when $|\tau|\le L^D$. The extra factor $R^{i \tau}$ is fixed with
$|R^{i\tau}|=1$. Thus we may take $\tau=0$, redefining $\psi_A(u)u^{i\tau}$ as
$\psi_A$.

Because $(v,\mathcal S)=1$ implies $(v,q)=1$, we may expand $b_A$
in characters modulo $q$. Every normalized Fourier coefficient has
absolute value at most one, so the expansion costs at most
$\varphi(q)\le q=L^{O(1)}$.
Fix $A\le E$ and a character and put $Q=R/A=x^{1+o(1)}$.
We will establish the estimate
\begin{equation}
 \frac1Q\sum_{(v,\mathcal S)=1}\chi(v)\psi_A(v/Q)
       \left\{d_j(v)-\frac{M_{j,\mathcal S}(v)}V\right\}
       \ll Q^{-\theta}L^{O(1)}.
 \label{eq:fixed-part-centered}
\end{equation}
Notice that if $x$ is sufficiently large, then for $v/Q$ to lie in the support of $\psi_A$ we must have $v$ in $[x^{1/2},x^2]$. Thus $M_{j,\mathcal S}(v)$ is exactly the residue
polynomial in \eqref{eq:density-def}, and
$M_{j,\mathcal S}(Qu)/V$ has logarithmic derivatives bounded by
fixed powers of $L$, by \eqref{eq:density-finite} and
\eqref{eq:record-derivatives}.

\paragraph{3. Cancelling the residues.}
We will estimate both terms in the left-hand side of
\eqref{eq:fixed-part-centered} and compare their main terms. For every fixed positive integer
$a$ and smooth $\psi$ with the stated support and
derivative bounds, we claim that
\begin{equation}
 \frac1Q\sum_{(v,\mathcal S)=1}d_a(v)\chi(v)\psi(v/Q)
 =\delta_\chi\int_0^\infty
            \psi(u)M_{a,\mathcal S}(Qu)\,du
       +O(Q^{-\theta}L^{O(1)}).
 \label{eq:rough-smooth-evaluation}
\end{equation}
The relevant Dirichlet series is
\begin{equation}
 \mathcal D_{a,\mathcal S}(s,\chi)
 :=\sum_{(v,\mathcal S)=1}\frac{d_a(v)\chi(v)}{v^s}
 =L(s,\chi)^a\prod_{p\in\mathcal S}(1-\chi(p)p^{-s})^a
       \qquad(\Re s>1).
 \label{eq:full-rough-character-series}
\end{equation}
We will shift the contour to $\Re s=1-\theta$. The choice $\theta=1/\log y_0$ will allow us to save $Q^{-\theta}$, while retaining the bound
$p^\theta\le e$ for every $p \in \mathcal S$. In particular, throughout
the strip to its right the finite Euler product is bounded by
\begin{equation}
 \prod_{p\in\mathcal S}(1+p^{-1+\theta})^a
 \le\exp\left(a\sum_{p\in\mathcal S}p^{-1+\theta}\right)
 \ll L^{C_a},
 \label{eq:full-rough-euler-bound}
\end{equation}
by Mertens' estimate. The elementary bounds
\cite[(24), (49)]{TaoNotes2} give, for every character,
\begin{equation}
 |L(1-\theta+it,\chi)|
       \ll q(\theta^{-1}+1+|t|).
 \label{eq:elementary-moving-line-growth}
\end{equation}
This elementary bound is sufficient because $q$ and $\theta^{-1}$
are bounded by fixed powers of $L$.

Write $\widehat\psi(s)=\int_0^\infty\psi(u)u^{s-1}\,du$.
Its transform satisfies
$|\widehat\psi(\sigma+it)|\ll_m L^{C_m}(1+|t|)^{-m}$
uniformly for $1-\theta\le\sigma\le2$. Mellin inversion gives
\[
 \frac1Q\sum_{(v,\mathcal S)=1}d_a(v)\chi(v)\psi(v/Q)
 =\frac1{2\pi i}\int_{(2)}
       \mathcal D_{a,\mathcal S}(s,\chi)
                         Q^{s-1}\widehat\psi(s)\,ds.
\]
Now shift to $\Re s=1-\theta$. The preceding growth bounds and Mellin
decay make the horizontal integrals vanish and bound the new
vertical integral by
\[
 Q^{-\theta}L^{O(1)}\int_{\mathbb R}(1+|t|)^{a-m}\,dt
       \ll Q^{-\theta}L^{O(1)}
       \qquad(m>a+1).
\]

For a nonprincipal character no pole is crossed. For the principal
character, every prime dividing $q$ already belongs to
$\mathcal S$, so its Euler factors combine to give exactly
\[
 \mathcal D_{a,\mathcal S}(s,\chi_0)
       =\zeta(s)^aF_{a,\mathcal S}(s).
\]
The residue contributed by the principal character is therefore
\[
\begin{aligned}
 &\Res_{s=1}\left\{\zeta(s)^aF_{a,\mathcal S}(s)
                               Q^{s-1}\widehat\psi(s)\right\}\\
 &\quad=\Res_{s=1}\int_0^\infty
        \psi(u)\zeta(s)^aF_{a,\mathcal S}(s)(Qu)^{s-1}\,du\\
 &\quad=\int_0^\infty\psi(u)\Res_{s=1}
        \left\{\zeta(s)^aF_{a,\mathcal S}(s)(Qu)^{s-1}\right\}\,du\\
 &\quad=\int_0^\infty\psi(u)M_{a,\mathcal S}(Qu)\,du,
\end{aligned}
\]
where the interchange is easily justified by the smoothness and compact support of $\psi$. The last equality follows from \eqref{eq:density-def}, since
$Qu\in[x^{1/2},x^2]$ on the support of $\psi$, where $M_{a,\mathcal S}$
agrees with the residue density. 

This proves
\eqref{eq:rough-smooth-evaluation}, which we now apply first with
$a=j$, $\psi=\psi_A$, and then with $a=1$,
$\psi(u)=\psi_A(u)M_{j,\mathcal S}(Qu)/V$.
Since $M_{1,\mathcal S}=V$, the two principal terms are both
\[
 \delta_\chi\int_0^\infty
             \psi_A(u)M_{j,\mathcal S}(Qu)\,du
\]
and thus cancel exactly, proving \eqref{eq:fixed-part-centered}.

\paragraph{4. Summing over $A$.}
We need to sum the bounds in \eqref{eq:fixed-part-centered} with weight $d_K(A)Q/R=d_K(A)/A$.
The character expansion costs only a fixed power of $L$ and
\[
 \sum_{\substack{A\le E\\A_{\mathcal S}=A}}\frac{d_K(A)}{A}
 \le\prod_{p\in\mathcal S}(1-p^{-1})^{-K}
 =V^{-K}\ll L^{O(1)}.
\]
Since $Q=R/A$ and every $A$ in the sum $\mathcal{R}_E$ we need to estimate has size at most $E$, we conclude that
\[
 \mathcal R_E(\tau)\ll L^{O(1)}(R/E)^{-\theta}
       \ll L^{C'}\exp(-cL^{3/5}).
\]
Combining this with \eqref{eq:centered-record-tail} proves
\eqref{eq:centered-mellin-bound}.
\end{proof}

\section{Estimating the comparison correlations}
\label{sec:arithmetic-bridge}

The key feature of the graph machinery is that it allows us to relate
the fixed-shift correlation involving
$\alpha_{\mathcal P}(n)\beta_{\mathcal P}(n+h)$
to averages of correlations involving
$\alpha_{\mathcal P}(n)\beta_{\mathcal P}(n+hD_1D_2)$,
where $D_1$ and $D_2$ are independent products of freely
chosen primes. The purpose of this section is to show
how this additional freedom allows us to establish
the desired cancellation in
Proposition~\ref{prop:centered-marked-correlation}
by means of a minor and major arc analysis.
The strategy is similar to that of
\cite[Sections~5.3--5.6]{OpenAIGraph}.

For this, we first need to state the precise consequence that will be derived from the graph results of the next section. We use the interval banks
described after \eqref{eq:banks}; for sufficiently large $x$ all their
primes avoid $h$. Fix an active family as in Proposition~\ref{prop:centered-marked-correlation}. In this
section $\mathcal G_s,\mathcal G_b$ denote its selected subfamilies,
$\mathcal G=\mathcal G_s\sqcup\mathcal G_b$, and all bank indices
range over $\mathcal G$. Write $\mathcal P_g=P_g$ and recall that $\alpha_{\mathcal P}$ and $\beta_{\mathcal P}$ are invariant
under multiplication by $p^a$, for any $p\in\mathcal P$ and any
integer $a\ge0$.

For a vector $\boldsymbol t=(t_g)$ of nonnegative integers, a $\boldsymbol t$-list at $n$ consists of $t_g$ ordered distinct primes dividing $n$ from each $\mathcal{P}_g$. For a bounded function $b$ of these lists, define
\begin{equation}
 W_{\boldsymbol t}^{b}(n)=
 \left(\prod_g\one_{\{\omega_g(n)=t_g\}}V_g^{-t_g}\right)
 \sum_{\substack{\boldsymbol p\ \text{a }\boldsymbol t\text{-list at }n}}
 b(\boldsymbol p),\qquad W_{\boldsymbol t}:=W_{\boldsymbol t}^{1}.
 \label{eq:bridge-mark-weight}
\end{equation}
For $t_g=0$ the empty tuple is the unique list. Thus
$W_{\boldsymbol1}=\Omega_{\mathcal P}$, the weight in the proposition.
Moreover, for fixed $m$, $0\le t_g\le m+1$ and $|b|\le1$,
\begin{equation}
 |W_{\boldsymbol t}^{b}(n)|
 \le\prod_g t_g!V_g^{-t_g}\le L^{C_m}.
 \label{eq:bridge-mark-cap}
\end{equation}

The graph argument gives a relation between the desired correlation
and a family of comparison correlations, each of which has two
independent products of auxiliary big primes in its shifts. These
products will make its Fourier multiplier small away from rational
arcs, while the estimates of Section~\ref{sec:analytic-services}
will control those arcs. We now state the relation directly in the
arithmetic notation needed below. The following estimate will be deduced in
Section~\ref{sec:graph}
from the results of \cite{OpenAIGraph}.

\begin{proposition}[The graph comparison]\label{prop:graph-comparison}
Fix $A>0$ and
$\Psi\in C_c^\infty((0,\infty))$. There are a fixed integer $m\ge1$, a constant $C>0$, a partition
$\mathcal G_b=\mathcal G_{b,1}\sqcup\mathcal G_{b,2}$ and a
finite family of comparison terms $\mathcal C_\nu$ specified below, with coefficients
$c_\nu \in \mathbb{R}$, such that
\begin{equation}
 \mathcal C_0+\sum_\nu c_\nu\mathcal C_\nu\ll_A L^{-A},
 \qquad
 \#\{\nu\}+\sum_\nu|c_\nu|\le L^C,
 \label{eq:graph-comparison-relation}
\end{equation}
where
\[
 \mathcal C_0=\sum_{n\ge1}\frac{\Psi(n/x)}n
       \alpha_{\mathcal P}(n)\beta_{\mathcal P}(n+h)
       W_{\boldsymbol1}(n)W_{\boldsymbol1}(n+h).
\]
Each comparison is specified as follows. Choose integers
$0\le r_{\nu,g}\le m$, zero for $g\in\mathcal G_s$, with at
least one positive $r_{\nu,g}$ for each big block $\mathcal G_{b,i}$ and put
$t_{\nu,g}=1+r_{\nu,g}$. For each $g$, independently draw
$r_{\nu,g}$ auxiliary primes from the probability law
$\mu_g(p)=1/(V_gp)$. Write
$\boldsymbol d_1,\boldsymbol d_2$ for the resulting tuples of primes corresponding respectively to $\mathcal G_{b,1}$ and $\mathcal G_{b,2}$
and write $D_1,D_2$ for their products. Then
\begin{equation}
\begin{aligned}
 \mathcal C_\nu={}&\mathbb E_{\boldsymbol d_1,\boldsymbol d_2}
       a_\nu(\boldsymbol d_1)b_\nu(\boldsymbol d_2)
       \sum_{w\ge1}\frac{\Psi(w/(xD_1D_2))}{w}\\
 &\quad\times \alpha_{\mathcal P}(w)\beta_{\mathcal P}(w+hD_1D_2)
       W_{\boldsymbol t_\nu}^{u_\nu}(w)
       W_{\boldsymbol t_\nu}^{v_\nu}(w+hD_1D_2),
 \label{eq:graph-comparison-term}
\end{aligned}
\end{equation}
where $a_\nu,b_\nu,u_\nu,v_\nu$ are functions of modulus at most one that are fixed for each $\nu$, with $u_\nu,v_\nu$ depending only on the components of the ordered list indexed by the big banks $\mathcal G_b$. The constants and the construction are uniform
over every active family $\mathcal{P}$ with $|\mathcal{G}_b|,|\mathcal{G}_s| \ge c J$, for any fixed choice of $c>0$.
\end{proposition}

\subsection{Two free products and the Fourier multiplier}

We shall show that each comparison in
\eqref{eq:graph-comparison-term} is smaller than every fixed negative
power of $L$. The relation \eqref{eq:graph-comparison-relation} will
then give the desired bound for $\mathcal C_0$.
We begin by applying Fourier inversion.
A bilinear estimate for the two independent auxiliary products
then shows that the resulting Fourier integral is negligible away
from neighborhoods of rationals with small denominators.
This reduces the problem to a major-arc analysis, which
will be carried out in the subsequent subsections.

Each auxiliary product contains between one and $O_m(J)$ primes
from the big banks $\mathcal G_b$. We therefore use dyadic scale
parameters $H_1,H_2$ satisfying
\begin{equation}
\begin{aligned}
 &\exp(L^{1/3}-O(1))\le H_i\le\exp(O_m(JL^{2/5}))
       \qquad(i=1,2),\\
 &H'=H_1H_2=x^{o(1)},\qquad Y=xH'=x^{1+o(1)}.
\end{aligned}
 \label{eq:bridge-scales}
\end{equation}
Fix a comparison and let $\mathcal C_\nu(H_1,H_2)$ be its contribution
from $H_i\le D_i<2H_i$ for $i=1,2$. There are only $L^{O(1)}$
pairs of scales to sum. The result will follow from a saving of any prescribed power of $L$, uniformly in each of them.

\paragraph{Separating the smooth weight.}
Given $H_1,H_2$, we thus restrict attention to $D_1D_2\asymp H'$. On this range, the support of $\Psi$ forces both $w$ and $w+hD_1D_2$ to have size comparable to $Y$. We wish to Fourier-expand $\Psi$, while preserving this localization. With this in mind, choose a fixed $\chi\in C_c^\infty((0,\infty))$ equal to one on a fixed compact interval containing all possible values of $w/Y$ and $(w+hD_1D_2)/Y$ on the support of $\Psi$. Then inserting $
\chi(w/Y)\chi((w+hD_1D_2)/Y)
$
leaves the relevant terms in $\mathcal C_\nu$ unchanged.
Put $\varphi(z)=\Psi(e^z)$ and
$\widehat\varphi(\sigma)=\int_{\mathbb R}\varphi(z)e^{-i\sigma z}\,dz$.
We have
\[
 \frac{\Psi(w/(xD_1D_2))}{w}
 =\frac1{2\pi Y}\int_{\mathbb R}\widehat\varphi(\sigma)
     (w/Y)^{-1+i\sigma}
     (D_1/H_1)^{-i\sigma}(D_2/H_2)^{-i\sigma}\,d\sigma.
\]
The functions $a_\nu(\boldsymbol d_1)$, $b_\nu(\boldsymbol d_2)$ currently depend on the lists $\boldsymbol d_i$ but we may simplify the analysis by averaging over all choices giving the same product $D_i$, so we define
\[
\begin{aligned}
 b_{1,\sigma}(d)&=\one_{\{H_1\le d<2H_1\}}
       (d/H_1)^{-i\sigma}
       \mathbb E\bigl(a_\nu(\boldsymbol d_1)\mid D_1=d\bigr),\\
 b_{2,\sigma}(d)&=\one_{\{H_2\le d<2H_2\}}
       (d/H_2)^{-i\sigma}
       \mathbb E\bigl(b_\nu(\boldsymbol d_2)\mid D_2=d\bigr).
\end{aligned}
\]
Both have modulus at most one. For simplicity of notation, let us define $\boldsymbol t=\boldsymbol t_\nu$, $b=u_\nu$ and
\begin{equation}
\begin{aligned}
 F_\sigma(w)&=\psi(w/Y)w^{i\sigma}
                  \alpha_{\mathcal P}(w)W_{\boldsymbol t}^{b}(w),
       &\psi(z)&=\chi(z)/z,\\
 G(v)&=\chi(v/Y)\beta_{\mathcal P}(v)
                         W_{\boldsymbol t}^{v_\nu}(v).
\end{aligned}
 \label{eq:bridge-U}
\end{equation}
We thus have
\[
 \mathcal C_\nu(H_1,H_2)=\frac1{2\pi Y}
       \int_{\mathbb R}\widehat\varphi(\sigma)Y^{-i\sigma}
                                      \mathcal I_\sigma\,d\sigma,
\]
where
\begin{equation}
 \mathcal I_\sigma=
 \mathbb E_{D_1,D_2}b_{1,\sigma}(D_1)b_{2,\sigma}(D_2)
       \sum_w F_\sigma(w)G(w+hD_1D_2).
 \label{eq:bridge-bilinear-correlation}
\end{equation}
Divisor moments and
\eqref{eq:bridge-mark-cap} give, uniformly in $\sigma$,
\begin{equation}
 \sum_w|F_\sigma(w)|^2+\sum_w|G(w)|^2\ll YL^{C_1}.
 \label{eq:bridge-energy}
\end{equation}
Consequently $|\mathcal I_\sigma|\ll YL^{C_1}$ by Cauchy--Schwarz. By the rapid decay of $\widehat\varphi$, the contribution of the region $|\sigma|>L$ to $\mathcal C_\nu(H_1,H_2)$ is $O_A(L^{-A})$ for every fixed $A>0$, uniformly in $H_1,H_2$. It therefore suffices to bound $\mathcal I_\sigma/Y$ uniformly for $|\sigma|\le L$. For the rest of the proof, fix such a $\sigma$ and write $F=F_\sigma$ and $b_i=b_{i,\sigma}$.

\paragraph{Localizing the Fourier multiplier.}
Writing $\widehat F(\theta)=\sum_nF(n)e(\theta n)$ and similarly for $G$, additive Fourier
inversion gives
\[
 \mathcal I_\sigma=\int_0^1
       \widehat F(\theta)\widehat G(-\theta)m(h\theta)\,d\theta,
\]
with
\[
 m(\theta)=\mathbb E b_1(D_1)b_2(D_2)e(\theta D_1D_2).
\]
The proof will proceed by controlling the local
energy of $\widehat F$ over the major arcs and
 bounding $m(h \theta)$ on the minor arcs. The following elementary product estimate gives the required minor-arc bound.

\begin{lemma}[{see~\cite[Lemma~4.2]{OpenAIGraph}}]
\label{lem:auxiliary-product}
Let $m \ge 1$ be a fixed integer and consider a collection of $O(J)$ disjoint prime banks. For the purposes of this lemma, these could be any disjoint sets of primes with $V_g\asymp1$. Let $D_1,D_2$ be independent products formed from
at most $m$ independent draws of law $\mu_g(p)=1/(V_gp)$ in each bank.
Here the numbers of draws are fixed before sampling but they may depend on the bank and be different for $D_1$ and $D_2$. Then, there exists a constant $C$, depending only on $m$ and on the implicit constants in the $O(J)$, $V_g\asymp1$ bounds above, such that
\begin{equation}
 \mathbb P(D_i=n)\le L^C/n.\label{eq:bridge-product-atom}
\end{equation}
Furthermore, if $b_i$ are functions bounded by one and supported on $[H_i,2H_i)$ for some $H_i\ge1$, then, for $(a,r)=1$ and
$|\theta-a/r|\le r^{-2}$, we have
\begin{equation}
 \left|\mathbb E b_1(D_1)b_2(D_2)e(\theta D_1D_2)\right|^2
 \ll L^C\left(\frac1r+\frac1{H_1}
       +\frac{\log(2r)}{H_2}
       +\frac{r\log(2r)}{H_1H_2}\right).
 \label{eq:bridge-multiplier-bound}
\end{equation}
\end{lemma}
\begin{proof}
Fix one product and write $k_g\le m$ for its number of draws
in bank $g$.
Unique factorization gives at most $\prod_g k_g!$ ordered
representations of a product $n$, each of probability
$n^{-1}\prod_gV_g^{-k_g}$. Since there are $O(J)$ banks, this
proves \eqref{eq:bridge-product-atom}.

For the second estimate, notice that if $a_n,b_t$ are complex coefficients supported on $[H_1,2H_1)$ and $[H_2,2H_2)$,
respectively, we have
\begin{equation}
 \left|\sum_{n,t}a_nb_te(\theta nt)\right|^2
 \ll\|a\|_2^2\|b\|_2^2
             (1+H_1/r)(H_2+r\log(2r)).
 \label{eq:integer-bilinear}
\end{equation}
Indeed, Cauchy--Schwarz in $t$, expansion of the square and the standard geometric-sum estimate bound
the left side by
\[
 \|a\|_2^2\|b\|_2^2
 \left(H_2+\sum_{1\le j\ll H_1}
                    \min(H_2,\|\theta j\|^{-1})\right).
\]
In a block of $j$-diameter less than $r/2$, the
rational parts $aj/r$ are separated modulo one by at least $1/r$,
whereas the perturbations $(\theta-a/r)j$ differ by at most
$1/(2r)$ within the block. Thus the
points $\theta j$ are separated by at least $1/(2r)$.
The sum of the minimum over one block is therefore
$O(H_2+r\log(2r))$. There are $O(1+H_1/r)$ blocks, proving
\eqref{eq:integer-bilinear}.

Finally put $a_n=\mathbb P(D_1=n)b_1(n)$ and
$b_t=\mathbb P(D_2=t)b_2(t)$. The support of $a_n, b_t$, (\ref{eq:bridge-product-atom}) and the fact that
both probability laws have mass one give
$\|a\|_2^2\ll L^C/H_1$ and $\|b\|_2^2\ll L^C/H_2$.
Substituting in \eqref{eq:integer-bilinear} and enlarging $C$ proves
\eqref{eq:bridge-multiplier-bound}.
\end{proof}

Let us apply this to our $m(\theta)$. For a prescribed $A_0>0$, choose a sufficiently large fixed $B_0$
and put $H=H'/L^{B_0}$. Dirichlet approximation supplies
\[
 1\le r\le Q:=\lfloor H'/L^{B_0-1}\rfloor,\qquad (a,r)=1,
 \qquad |\theta-a/r|\le(rQ)^{-1}\le r^{-2}.
\]
If $r\le L^{B_0}$, this places $\theta$ within $H^{-1}$ of a rational
with denominator at most $L^{B_0}$. Otherwise
\eqref{eq:bridge-multiplier-bound} gives
\[
 |m(\theta)|^2\ll L^C
       \bigl(L^{-B_0}+H_1^{-1}+L/H_2+L^{2-B_0}\bigr)
       \ll L^{-2A_0},
\]
by our choice of $B_0$. The bounds in \eqref{eq:bridge-scales}
make the middle two terms smaller than every fixed negative power
of $L$.

For $m(h\theta)$, the exceptional arcs are therefore contained in
\[
 \mathfrak M=
 \bigcup_{\substack{1\le q\le |h|L^{B_0}\\0\le a<q,\ (a,q)=1}}
 \left\{\theta\in\mathbb R/\mathbb Z:
                \|\theta-a/q\|\le C_h/H\right\}.
\]
Here $C_h$ is a fixed constant depending on $h$. Thus, there are
$O_h(L^{2B_0})$ arcs and $|m(h\theta)|\ll L^{-A_0}$ on their
complement. Every prime dividing a denominator $q$ belongs to
$\mathcal S$ for sufficiently large $x$.

\paragraph{The remaining local estimate.}
Parseval, Cauchy--Schwarz and \eqref{eq:bridge-energy} combined with our minor arc bound give
\[
 \frac1Y\left|\int_{[0,1)\setminus\mathfrak M}
       \widehat F(\theta)\widehat G(-\theta)m(h\theta)\,d\theta\right|
       \ll L^{C_1-A_0}.
\]
On an individual arc $I_{a/q}$, we instead use:
\[
\begin{aligned}
 \frac1Y\int_{I_{a/q}}
       |\widehat F(\theta)\widehat G(-\theta)m(h\theta)|\,d\theta
 &\le\frac1Y\left(\int_{I_{a/q}}|\widehat F(\theta)|^2d\theta\right)^{1/2}
                \left(\sum_w|G(w)|^2\right)^{1/2}\\
 &\ll L^{C_1/2}\left(\frac1Y\int_{|\beta|\le C_h/H}
       |\widehat F(a/q+\beta)|^2d\beta\right)^{1/2}.
\end{aligned}
\]
Thus the complement is negligible when $A_0$ is large. Summing over
the $O_h(L^{2B_0})$ arcs, it suffices to prove, for every required
fixed $N$ and every choice of $a,q$ in our range,
\begin{equation}
 \int_{|\beta|\le C_h/H}
 |\widehat F(a/q+\beta)|^2\,d\beta\ll_N YL^{-N}.
 \label{eq:bridge-local-energy-target}
\end{equation}
We choose $N$ after $B_0$, so that this saving pays for the number
of arcs and the additional logarithmic factors. This is the remaining goal of the section.

\subsection{Passing to Dirichlet polynomials}

We shall use the following standard estimate to pass from additive to multiplicative phases.

\begin{lemma}
\label{lem:local-mellin-energy}
Let $C>0$. If $a_n$ is supported on a fixed range $n\asymp Y$,
$Y=x^{1+o(1)}$, and $H=x^{o(1)}\to\infty$, then
\begin{equation}
 \int_{|\beta|\le CH^{-1}}\left|\sum_na_ne(\beta n)\right|^2d\beta
 \ll_C \frac1Y\int_{\mathbb R}(1+H|t|/Y)^{-2}
           \left|\sum_na_nn^{it}\right|^2dt.
 \label{eq:bridge-local-mellin}
\end{equation}
\end{lemma}
\begin{proof}
Put $\eta=H/(4C)$. Plancherel, applied to the interval kernel of
length $\eta$, bounds the left side by
\[
 \ll \eta^{-2}\int_0^\infty
       \left|\sum_{v<n\le v+\eta}a_n\right|^2dv.
\]
By \cite[Lemma~2.2]{SoundararajanBourbaki}, this is in turn bounded by
\[
 \ll \eta^{-2}Y\int_{\mathbb R}
  \min\{\eta^2/Y^2,|t|^{-2}\}
       \left|\sum_na_nn^{it}\right|^2dt,
\]
which proves the claim. 
\end{proof}

We shall also use the standard mean-value theorem for Dirichlet
polynomials \cite[Corollary~3]{MontgomeryVaughanHilbert}:
uniformly over intervals $I$ of length $T$ and coefficients supported on $n=O(R)$,
\begin{equation}
 \int_I\left|\sum_{n}a_nn^{it}\right|^2dt
       \ll(T+R)\sum_n|a_n|^2.
 \label{eq:bridge-elementary-mean-square}
\end{equation}
Apply Lemma~\ref{lem:local-mellin-energy} with
$a_n=F(n)e(an/q)$. Choose a sufficiently large fixed $B_1$ and put
\begin{equation}
 T_0=(Y/H)L^{B_1}=xL^{B_0+B_1},
 \label{eq:bridge-T0}
\end{equation}
with $B_0$ coming from our definition of the major arcs. By \eqref{eq:bridge-energy} and
\eqref{eq:bridge-elementary-mean-square},
\[
Y^{-2}\int_{T\le|t|\le2T}
 \left|\sum_nF(n)e(an/q)n^{it}\right|^2dt
 \ll L^{C_1}(1+T/Y).
\]
On the successive dyadic intervals beyond $T_0$, summing this
estimate with the corresponding weights gives
\[
\frac1{Y^2}\int_{|t|>T_0}(1+H|t|/Y)^{-2}
 \left|\sum_nF(n)e(an/q)n^{it}\right|^2dt
 \ll L^{C_1}\left(\frac{1}{L^{2B_1}}+\frac1{HL^{B_1}}\right).
\]
Choosing $B_1$
sufficiently large makes this smaller than any prescribed negative power of $L$.
It thus only remains to establish the bound
\begin{equation}
 \int_{|t|\le T_0}
 \left|Y^{-1}\sum_nF(n)e(an/q)n^{it}\right|^2dt\ll_N L^{-N}.
 \label{eq:bridge-Mellin-target}
\end{equation}

\subsection{The low frequencies}

Write
\[
 \mathcal T(t)=Y^{-1}\sum_n F(n)e(an/q)n^{it},
\]
and let $B>0$ be a fixed large constant to be specified in the next subsection. We will first control the range $|t|\le T_c$ with $T_c=L^B$, where we shall exploit the fact that $\alpha_{\mathcal{P}}$ has been centered by construction. Indeed, expanding the definitions of $F$ and $\alpha_{\mathcal{P}}$ given respectively in \eqref{eq:bridge-U} and \eqref{eq:record-centered-core} and writing $A=n_{\mathcal S}$, $v=n/A$, we see that
\begin{equation}
\begin{aligned}
 \mathcal T(t)
  =\frac1Y\sum_{A=A_{\mathcal S}}H(A)
    \sum_{(v,\mathcal S)=1}
       e(aAv/q)\psi(Av/Y)(Av)^{i(t+\sigma)}
       \left\{d_k(v)-\frac{M_{k,\mathcal S}(v)}V\right\},\\
 H(A)=d_k(A_{\mathcal S\setminus\mathcal P})
                                      W_{\boldsymbol t}^{b}(A),
 \qquad |H(A)|\le L^C d_k(A).
\end{aligned}
 \label{eq:bridge-low-record-form}
\end{equation}
Notice we are using the fact that $W_{\boldsymbol t}^{b}(n)$ depends only on $n_{\mathcal S}$.  Furthermore, $q$ is $\mathcal S$-smooth. We may thus apply Lemma~\ref{lem:combined-mellin} with $R=Y$, $j=K=k$,
$b_A(v)=e(aAv/q)$ and $\psi_A=\psi$. Since $|H(A)|\le L^Cd_k(A)$
and $|\sigma|\le L^{O(1)}$, the triangle inequality gives,
uniformly for $|t|\le T_c$,
\[
 |\mathcal T(t)|
   \le\frac{L^C}{Y}\sum_{A_{\mathcal S}=A}
                      d_k(A)|S_A(t+\sigma)|
   \ll L^{C_B}\exp(-c_BL^{1/10}).
\]
Hence, for every fixed $N$,
\begin{equation}
 \int_{|t|\le T_c}|\mathcal T(t)|^2\,dt\ll_N L^{-N},
 \label{eq:bridge-low-times}
\end{equation}
where the constants in this estimate may depend on the subsequently fixed
$B$.

\subsection{Extracting a prime sum at high frequencies}

For $T_c<|t|\le T_0$, we extract the unique prime
divisor from a small bank and use cancellation in
the resulting prime sum, while controlling the
remaining factor by a mean-square estimate.
This is in the spirit of \cite{MR}
(see also \cite[Section~3]{SoundararajanBourbaki}), although
here the reduction in Section~\ref{sec:reconstruction}
allows us to restrict to exactly one distinct prime
divisor in each active small bank. This makes the
prime-factor decomposition immediate apart from
the prime-square exceptions treated below.

Fix a small bank $g_s \in \mathcal{G}_s$. Its list in $\boldsymbol t$ has length one and the
function $b$ does not depend on the primes in this bank. Except when a prime of this
group divides $n$ to the second power, one easily checks from the definitions that
\begin{equation}
 W_{\boldsymbol t}^{b}(n)=\frac1{V_{g_s}}
    \sum_{\substack{p\mid n\\p\in\mathcal P_{g_s}}}
       W_{\boldsymbol t-\boldsymbol e_{g_s}}^{b}(n/p).
 \label{eq:bridge-owner-identity}
\end{equation}
Since
$\alpha_{\mathcal P}(pm)=\alpha_{\mathcal P}(m)$ for every $p \in \mathcal{P}$,
putting
\[
 \gamma(m)=\alpha_{\mathcal P}(m)
                 W_{\boldsymbol t-\boldsymbol e_{g_s}}^b(m)
\]
gives the explicit expression
\[
 \mathcal T(t)=\frac1{YV_{g_s}}
   \sum_{p\in\mathcal P_{g_s}}\sum_{m\ge1}
       \gamma(m)\psi(pm/Y)e(apm/q)(pm)^{i(t+\sigma)}
       +\mathcal E(t),
\]
where the condition $\omega_{g_s}(m)=0$ in the support of $\gamma$ makes
$p\nmid m$ and $\mathcal E$ is the part of $\mathcal{T}$ that comes from integers divisible by a square from $\mathcal{P}_{g_s}$. Let $\mathcal E_{g_s}$ denote the set of such integers. Since there are $O(Y\sum_{p\in\mathcal P_{g_s}}p^{-2})
\ll Y\exp(-cL^{1/8})$ such integers with $n\asymp Y$, the
Cauchy--Schwarz inequality, fourth divisor moments and
\eqref{eq:bridge-mark-cap} give
\[
\sum_{n\in\mathcal E_{g_s}}|F(n)|^2
\ll
YL^C \exp(-cL^{1/8}).
\]
Consequently, since $T_0/Y=o(1)$, it follows from 
\eqref{eq:bridge-elementary-mean-square} that
\[
 \int_{|t|\le T_0}|\mathcal E(t)|^2\,dt
          \ll L^C\exp(-cL^{1/8}).
\]

\paragraph{Separating the variables.} 
In order to finish decoupling the variables $p$ and $m$ it remains to deal with $\psi(pm/Y)$ and $e(apm/q)$. For the latter, we just restrict $p\equiv r\pmod q$ and $m\equiv r'\pmod q$.
The additive phase is then the constant $e(arr'/q)$. For the former, we first localize the variables around dyadic intervals. We use a smooth dyadic partition $1=\sum_P f(p/P)$, with $f$ fixed
and supported in $[1/2,2]$. Once a dyadic scale $P$ is fixed let us write $R_0=Y/P$. On the support of $f(p/P)\psi(pm/Y)$, the ratio
$m/R_0$ lies in a fixed compact subinterval of $(0,\infty)$.
We may therefore insert a fixed smooth cutoff $g(m/R_0)$ equal
to one there.

The scales satisfy
\begin{equation}
 L^{1/8}-O(1)\le\log P\le L^{1/4}+O(1),\qquad
 R_0=x^{1+o(1)},\qquad
 \frac{T_0}{R_0}\ll\frac{PL^{B_0+B_1}}{H'}=o(1),
 \label{eq:bridge-residual-scales}
\end{equation}
where we recall that $T_0$ was specified in (\ref{eq:bridge-T0}) and the last inequality follows because $H'$ contains two nonempty
big-prime products, whereas $P$ is at the small-prime scale.

Let us now define
\[
 \widetilde\psi(\nu)=\int_0^\infty\psi(u)u^{-i\nu}\,\frac{du}{u}.
\]
Fourier inversion then gives
\[
 \psi(pm/Y)=\frac1{2\pi}\int_{\mathbb R}
       \widetilde\psi(\nu)Y^{-i\nu}p^{i\nu}m^{i\nu}\,d\nu.
\]
Hence, writing $s=(P,r,r',\nu)$, we obtain
\begin{equation}\label{eq:bridge-mellin-decomposition}
 \mathcal T(t)-\mathcal E(t)
 =\frac1{2\pi V_{g_s}}\sum_P\sum_{r,r'\bmod q}e(arr'/q)
    \int_{\mathbb R}\widetilde\psi(\nu)Y^{-i\nu}
                         P_s(t)R_s(t)\,d\nu,
\end{equation}
where
\begin{equation}
\begin{aligned}
 P_s(t)&=\frac1P\sum_{\substack{p\in\mathcal P_{g_s}\\p\equiv r\pmod q}}
               f(p/P)p^{i(t+\sigma+\nu)},\\
 R_s(t)&=\frac1{R_0}\sum_{m\equiv r'\pmod q}
               g(m/R_0)\gamma(m)m^{i(t+\sigma+\nu)}.
\end{aligned}
 \label{eq:bridge-factorized}
\end{equation}
The factor $1/Y$ has split exactly as $1/(PR_0)$. There are
only $L^{O(1)}$ choices of $P,r,r'$, and
$\widetilde\psi$ is a fixed Schwartz function. 

The pointwise bounds for $\alpha_{\mathcal P}$ and $W$, together
with divisor moments, give
$\sum_{m\asymp R_0}|\gamma(m)|^2\ll R_0L^{C}$.
Thus \eqref{eq:bridge-elementary-mean-square} gives, uniformly in
$\nu$ and in the location of any interval $I$ of length at most
$4T_0$,
\begin{equation}
 \int_I|R_s(t)|^2\,dt\ll L^{C}.
 \label{eq:bridge-residual-mean-square}
\end{equation}
Combining this with the trivial bound $|P_s(t)|\ll1$ and
Minkowski's inequality, the rapid decay of $\widetilde\psi$
allows us to discard $|\nu|>L$ through the bound
\[
\int_{|t|\le T_0}
\left|
\int_{|\nu|>L}
\widetilde\psi(\nu)Y^{-i\nu}
P_s(t)R_s(t)\,d\nu
\right|^2dt
\ll_N L^{-N},
\]
for each fixed $P, r, r'$. Summing over all $\ll L^{O(1)}$ such choices, this gives the desired estimate in this range.

\paragraph{Smoothing the banks.}
To deal with the remaining range we want to apply Lemma~\ref{lem:zf-smooth-prime}, but for this we first need to smooth the sharp restriction $p\in\mathcal P_{g_s}$ in a standard way. In each dyadic piece, choose a cutoff $0\le w(p/P)\le1$ supported inside the bank and equal to one except in endpoint intervals of the bank of total length $O(\varepsilon P)$, where $\varepsilon=L^{-K}$. Let $P_s^{\rm sm}$ be the prime sum with this additional cutoff. Counting integers in the transition intervals gives, uniformly in $t$ and $\nu$,
\[
 |P_s(t)-P_s^{\rm sm}(t)|\ll\varepsilon+P^{-1}.
\]
By \eqref{eq:bridge-residual-mean-square}, the resulting integrated squared error is $
O\!\left(L^C(\varepsilon+P^{-1})^2\right)$.
Since there are only $L^{O(1)}$ residue classes and dyadic pieces, this remains smaller than any prescribed negative power of $L$ if $K$ is chosen sufficiently large. For this fixed choice of $K$, the derivatives of $fw$ cost only fixed powers of $L$.

\paragraph{The high-frequency bound.}
The smoothed prime weight $fw$ is now supported inside the bank,
so its sharp bank restriction can be omitted. All its primes
exceed $q$. Expanding the residue-class indicator in characters
modulo $q$ and applying Lemma~\ref{lem:zf-smooth-prime} gives,
for each fixed $j\ge1$,
\[
 |P_s^{\rm sm}(t)|
   \ll_j L^{C_{j,K}}(1+|t+\sigma+\nu|)^{-j}
       +L^{C_K}\exp(-cL^{1/8}),\qquad |t|\le T_0.
\]
Recall that $|\sigma|, |\nu| \le L$. Let $A_s > 0$ be sufficiently large and let $B>1$ be sufficiently large with respect to $A_s$ and $K$. For $|t|>T_c=L^B$, we have
$|t+\sigma+\nu|\ge |t|/2$ and the preceding bound is at most
$L^{-A_s}$. It then follows that
\begin{equation}
\begin{aligned}
 \int_{T_c<|t|\le T_0}|P_s^{\rm sm}(t)R_s(t)|^2\,dt
 &\le \sup_{T_c<|t|\le T_0}|P_s^{\rm sm}(t)|^2
               \int_{|t|\le T_0}|R_s(t)|^2\,dt\\
 &\ll L^{-A_s}.
\end{aligned}
 \label{eq:bridge-high-times}
\end{equation}
We can now combine this with Minkowski's inequality to bound \eqref{eq:bridge-mellin-decomposition}. The $L^1$ norm of $\widetilde\psi$ is $O(1)$ and the
number of terms costs only a fixed power of $L$. Including the
square exceptions, Fourier tails and smoothing errors already
estimated, the high-frequency contribution is thus $O_N(L^{-N})$.
Together with \eqref{eq:bridge-low-times}, this proves
\eqref{eq:bridge-Mellin-target}. Returning through the
preceding reductions, each comparison correlation
$\mathcal C_\nu$ therefore has an arbitrary fixed
logarithmic saving. Using the second inequality in
\eqref{eq:graph-comparison-relation} and choosing this saving
sufficiently large, we obtain
\[
 \sum_\nu c_\nu\mathcal C_\nu\ll_A L^{-A}.
\]
The first inequality in \eqref{eq:graph-comparison-relation}
then gives the required bound for $\mathcal C_0$.
This proves Proposition~\ref{prop:centered-marked-correlation}.

\section{The comparison graphs}
\label{sec:graph}

In this final section, we prove Proposition~\ref{prop:graph-comparison} by applying the
weighted-dilation construction of \cite{OpenAIGraph}.
We begin by recalling the graph framework of that paper
in slightly greater generality than is strictly needed
for our application, in order to make the relevant
definitions and normalizations clear.
We then show how the estimates of \cite{OpenAIGraph}
adapt to the specific complete prime-divisor lists that our arguments require
and how they yield the desired arithmetic comparison relation. We remark that Lean formalizations of the intermediate graph results
we use from \cite{OpenAIGraph} are also available
in~\cite{XuTotientLean}.

\subsection{The represented divisibility graph}

Fix an active family as in Section~\ref{sec:arithmetic-bridge}.
Write $\mathcal G=\mathcal G_s\sqcup\mathcal G_b$ for its index
set, with $\mathcal G_s$ indexing the small banks and
$\mathcal G_b$ the big banks. The banks $\mathcal P_g=P_g$
are pairwise disjoint intervals of primes satisfying the
prime-range and reciprocal-mass bounds in \eqref{eq:banks}.
Their index sets satisfy
\[
 c_0J\le |\mathcal G_s|,|\mathcal G_b|\le C_*J
\]
for a fixed $c_0>0$, as in Proposition~\ref{prop:graph-comparison}.
Recall that $\mu_g(p)=(V_gp)^{-1}$ is a probability law on
$\mathcal P_g$. We write
\[
 \mathcal P=\bigcup_{g\in\mathcal G}\mathcal P_g,\qquad
 s=|\mathcal G|,\qquad
 \omega_g(n)=\#\{p\in\mathcal P_g:p\mid n\}.
\]
Unless a smaller range is specified, bank indices range over
$\mathcal G$. A subscript $b$ on a label tuple denotes its
restriction to $g\in\mathcal G_b$. For large $x$ no bank prime
divides the fixed nonzero shift $h$. All statements below are
uniform over these active families.

\paragraph{Represented vertices.}
We will work with several graphs constructed over the same vertex set, which we now define. Fix a positive integer $M$. A represented vertex
is a pair $v=(n,\boldsymbol p)$, where $n\in\mathbb Z$ and
\[
 \boldsymbol p=(\boldsymbol p_g)_{g\in\mathcal G},\qquad
 \boldsymbol p_g=(p_{g,1},\ldots,p_{g,M})\in\mathcal P_g^M,
\]
such that the entries in each bank are distinct and divide $n$.
These ordered prime lists are the \emph{marks} of the vertex.
They need not include every bank prime divisor of $n$. A list in
bank $g$ is \emph{complete} if it includes all these divisors,
equivalently if $\omega_g(n)=M$.

Let $\Omega_M^{\rm rep}$ be the set of represented vertices. Give
each vertex mass $\lambda_M=\prod_gV_g^{-M}$ and use
\[
 \mathcal H_M^{\rm rep}=\ell^2(\Omega_M^{\rm rep},\lambda_M),
 \qquad
 \langle f,g\rangle=\sum_{v\in\Omega_M^{\rm rep}}
                              \lambda_M\overline{f(v)}g(v).
\]
For $1\le r<\infty$, the corresponding norms are
$\|f\|_r^r=\sum_{v\in\Omega_M^{\rm rep}}\lambda_M|f(v)|^r$.
A represented vertex has $\omega_g(n)\ge M$ in every bank, so if we define its
\emph{endpoint damping} by
\begin{equation}
 \delta_M(n)=2^{-\frac12\sum_g(\omega_g(n)-M)}\le1
 \label{eq:graph-damping}
\end{equation}
then it equals one precisely when all its lists are complete.
The damping compensates for the growing number of represented
lists when an integer has more than $M$ prime divisors in a bank;
it disappears on the complete-list vertices used in our application.

\paragraph{Patterns.}
A \emph{pattern} is the following pair of data:
\[
 \bigl((C_g)_{g\in\mathcal G},K\bigr),\qquad
 C_g\subseteq\{1,\ldots,M-1\},\qquad
 C_g=\varnothing\quad(g\in\mathcal G_s),
\]
where $K$ is a bounded complex-valued function with domain
\[
 K:\prod_{g\in\mathcal G_b}
       \bigl(\mathcal P_g^M\times\mathcal P_g^M
                              \times\mathcal P_g^{C_g}\bigr)
                                               \longrightarrow\mathbb C.
\]
We write its arguments as
$K(\boldsymbol p_b,\boldsymbol q_b,\boldsymbol r_b)$.
This full domain includes repeated prime values.

The slots in $C_g$ are called \emph{freed}. We say the slots $j<M$
outside $C_g$ are \emph{shared} and the slots in
$C_g\cup\{M\}$ are \emph{unshared}. Put $t_g=1+|C_g|$, the
number of unshared slots in bank $g$. Write
$\boldsymbol t=(t_g)_{g\in\mathcal G}$. The last slot $M$ is always
unshared. For a pattern indexed by $\nu$, we use $C_g^{(\nu)}$
and $K_\nu$, suppressing $\nu$ when discussing a single pattern.

\paragraph{The graph of one pattern.}
Fix a pattern $\nu$ and a real phase parameter $\zeta$. For
represented vertices $v=(n,\boldsymbol p)$ and
$w=(n',\boldsymbol q)$ and an auxiliary tuple
$\boldsymbol r\in\prod_g\mathcal P_g^{C_g}$, define
\begin{equation}
 D_g=\prod_{\substack{j<M\\j\notin C_g}}p_{g,j}
             \prod_{j\in C_g}r_{g,j},\qquad
 D=\prod_gD_g,\qquad D_b=\prod_{g\in\mathcal G_b}D_g.
 \label{eq:graph-displacement-product}
\end{equation}
Each $D_g$ has $M-1$ prime factors, counted with multiplicity.
The labelled edge set $E_{h,\nu}$ consists of the triples
$e=(v,w,\boldsymbol r)$ in
$(\Omega_M^{\rm rep})^2\times\prod_g\mathcal P_g^{C_g}$
satisfying
\[
\begin{aligned}
 n'&=n+hD,\\
 q_{g,j}&=p_{g,j} &&(j<M,\ j\notin C_g),\\
 r_{g,j}&\notin\{p_{g,i},q_{g,i}:1\le i\le M\}
                                                   &&(j\in C_g).
\end{aligned}
\]
We call $v$ its source and $w$ its target. We write $n(e)=n$,
$n'(e)=n'$ and $D(e)=D$ for its integer coordinates and displacement
product. No further equality between the endpoint lists is imposed
at unshared slots. Auxiliary labels may equal one another. Different auxiliary tuples specify
different edges even if the endpoints $v, w$ agree.

Give $e\in E_{h,\nu}$ the weight
\begin{equation}
 w_{\nu,\zeta}(e)=
 \left(\prod_gV_g^{-t_g}\right)
 \left(\prod_g\prod_{j\in C_g}\mu_g(r_{g,j})\right)
 K_\nu(\boldsymbol p_b,\boldsymbol q_b,\boldsymbol r_b)D^{i\zeta}
                    \delta_M(n)\delta_M(n').
 \label{eq:represented-edge-weight}
\end{equation}
The \emph{represented divisibility graph} of this pattern is the
weighted directed multigraph
\[
 \Gamma_{h,\nu,\zeta}^{\rm rep}
   =\bigl(\Omega_M^{\rm rep},E_{h,\nu},w_{\nu,\zeta}\bigr),
\]
with vertex measure $\lambda_M$. Its adjacency operator is
\[
 (\widetilde A_{h,\nu,\zeta}^{\rm rep}f)(v)
       =\sum_{e=(v,w,\boldsymbol r)\in E_{h,\nu}}
                                             w_{\nu,\zeta}(e)f(w).
\]
Notice this finite sum includes all admissible auxiliary tuples. 

\paragraph{Signed families and permutation averages.}
Let $\mathfrak F$ be a finite index set containing $0$. A signed
comparison family specifies a pattern
$((C_g^{(\nu)})_g,K_\nu)$ for each $\nu\in\mathfrak F$, with $\nu=0$ always given by the \emph{raw} member $(C_g^{(0)}, K_0)=(\varnothing, 1)$. The
members $\nu \neq 0$ are called \emph{comparisons}.  We shall construct such a family with 
\begin{equation}
 |C_g^{(\nu)}|\le m,\qquad
              \sum_{\nu\in\mathfrak F}\sup|K_\nu|\le L^{C_1},
 \label{eq:g2}
\end{equation}
where the supremum is over the full domain of $K_\nu$. The integer
$m\ge1$ and constant $C_1>0$ are chosen before a sufficiently
large fixed $M$.

Let $\mathfrak S_M$ be the group of permutations of
$\{1,\ldots,M\}$. For
$\boldsymbol\sigma=(\sigma_g)_{g\in\mathcal G}\in
\mathfrak S_M^{\mathcal G}$, put
$p^{\boldsymbol\sigma}_{g,j}=p_{g,\sigma_g(j)}$ and define
\[
 (\Pi f)(n,\boldsymbol p)
   =\frac1{(M!)^s}\sum_{\boldsymbol\sigma\in\mathfrak S_M^{\mathcal G}}
                                 f(n,\boldsymbol p^{\boldsymbol\sigma}).
\]
Thus $\Pi$ averages the list orders independently in every bank.
The vertex measure is invariant under these permutations, so $\Pi$
is an orthogonal projection. The family's symmetrized adjacency
operator is
\begin{equation}
 A_{h,\zeta}^{\rm rep}
       =\Pi\left(\sum_{\nu\in\mathfrak F}
                         \widetilde A_{h,\nu,\zeta}^{\rm rep}\right)\Pi.
 \label{eq:represented-signed-operator}
\end{equation}
The two projections average the list orders at the source and target
independently. The family and its coefficients will remain fixed as $\zeta$ varies.

\subsection{The independent model}

We now recall the independent model from \cite{OpenAIGraph}.
This is not strictly speaking necessary for our needs, since we could
invoke the relevant estimates from that paper directly.
We nevertheless believe that making the model and its
operator explicit will make it substantially easier for
the reader to follow their application in our setting. 

The independent model keeps the pattern
coefficients but forgets about the small-prime list, while for the big-prime list it replaces divisibility by independent sampling. In particular, it does not distinguish integers $n \in \mathbb{Z}$.
Its label space, probability measure and Hilbert space are
\[
 \mathcal L_b=\prod_{g\in\mathcal G_b}\mathcal P_g^M,
 \qquad \rho_b=\bigotimes_{g\in\mathcal G_b}\mu_g^{\otimes M},
 \qquad \mathcal H_b=L^2(\mathcal L_b,\rho_b).
\]
Repeated entries are allowed in $\mathcal L_b$. Given a pattern
$\nu$ and a source tuple $\boldsymbol p_b$, copy each shared
coordinate to the target: $q_{g,j}=p_{g,j}$ for $j<M$ outside
$C_g^{(\nu)}$. Draw all unshared target coordinates and all
auxiliary coordinates independently with their respective laws
$\mu_g$. For $\theta,\zeta\in\mathbb R$, define
\begin{equation}
 (\widetilde T_{\theta,\nu,\zeta}F)(\boldsymbol p_b)
   =\mathbb E_{\boldsymbol r_b,\boldsymbol q_b}
       K_\nu(\boldsymbol p_b,\boldsymbol q_b,\boldsymbol r_b)
       D_b^{i\zeta}e(\theta D_b)F(\boldsymbol q_b),
 \qquad e(t)=\exp(2\pi it).
 \label{eq:g3}
\end{equation}
Here $D_b$ is the product in
\eqref{eq:graph-displacement-product}, evaluated at the sampled
labels. The expectation imposes no distinctness or auxiliary-label
exclusion. Let $\Pi_b$ average uniformly over all permutations of each
big-bank list, independently from bank to bank. The signed
independent operator is
\begin{equation}
 T_{\theta,\zeta}
   =\Pi_b\left(\sum_{\nu\in\mathfrak F}
                  \widetilde T_{\theta,\nu,\zeta}\right)\Pi_b.
 \label{eq:independent-signed-operator}
\end{equation}
We can now compare our setting with
\cite[Definitions~3.1--3.3]{OpenAIGraph}. In their notation, we take $\ell=1$,
their fixed shared-list parameter equal to $M-1$ and their
bank-count parameter $T$ equal to our $J=\log L$.
We also take their damping parameter $q=1/2$ and
$(a,b,c,d)=(1/8,1/4,1/3,2/5)$, which satisfy
$0<a<b<c<d<0.47$.
With these choices, \eqref{eq:represented-signed-operator}
is their physical operator from Definition~3.3, while
\eqref{eq:independent-signed-operator} is their ideal operator
from Definition~3.2.

\paragraph{The structured comparison family.}
We now describe the precise signed comparison family constructed in
\cite{OpenAIGraph}. We will soon use its specific structure
to establish the properties of the comparison terms asserted
in Proposition~\ref{prop:graph-comparison}. 

Partition the big-bank index set as
$\mathcal G_b=\mathcal G_1\sqcup\mathcal G_2$, with the two
sizes differing by at most one. For an integer $1\le m<M$, designate the first $m$ slots
in each big bank as \emph{probes}: these are the slots eligible
to be freed in a comparison. Thus, the sets $C_g^{(\nu)}$ with $\nu \neq 0$
will range over all choices of subsets of $\{1,\ldots,m\}$ for each $g\in\mathcal G_b$,
subject to at least one slot being freed in each of the two
blocks. Write
\[
 \mathcal Q_i=\{(g,j):g\in\mathcal G_i,\ 1\le j\le m\}
 \qquad(i=1,2)
\]
for the index sets of the probes in the two blocks. We now define the coefficients $K_{\nu}$. Intuitively, their values at given lists of primes will provide a way of measuring how close the relevant products constructed from these primes are, not only in size but also on average over residue classes of small moduli. For this, fix $Q\ge1$ and
$0<\delta,\xi\le1$. For a nonempty slot set
$\Lambda\subseteq\mathcal Q_1\cup\mathcal Q_2$, take
independently chosen $R_{g,j}$ according to the probability laws $\mu_g$. Define the random product,
logarithmic cells and cell masses by
\begin{equation}
\label{12}
 Z_\Lambda=\prod_{(g,j)\in\Lambda}R_{g,j},\qquad
 I_k=[k\delta,(k+1)\delta),\qquad
 \nu_{\Lambda k}=\mathbb P(\log Z_\Lambda\in I_k)
 \quad(k\in\mathbb Z).
\end{equation}
Let $\mathfrak X_Q$ be the set of primitive Dirichlet characters
of conductor at most $Q$, including the conductor-one character.
For positive integers $D,X$ define, as in \cite[(4.7)]{OpenAIGraph},
\begin{equation}
 \mathcal K_\Lambda(D,X)=
 \sum_{k:\nu_{\Lambda k}\ge\xi}
 \frac{\one_{\{\log D\in I_k\}}\one_{\{\log X\in I_k\}}}
      {\nu_{\Lambda k}}
       \sum_{\chi\in\mathfrak X_Q}\overline{\chi(D)}\chi(X).
 \label{eq:separated-kernel-formula}
\end{equation}
Only finitely many cells have positive mass, so this is a finite
sum. The construction chooses $Q,\delta^{-1},\xi^{-1}$ as fixed
powers of $L$ before choosing $M$; see
\cite[Lemma~4.3]{OpenAIGraph}.

Index the comparisons by pairs of nonempty subsets
$A\subseteq\mathcal Q_1$, $B\subseteq\mathcal Q_2$. For
$\nu=(A,B)$, let
\[
\begin{aligned}
 C_g^{(\nu)}&=\{j:(g,j)\in A\cup B\},\\
 D_A&=\prod_{(g,j)\in A}r_{g,j},&
 X_A&=\prod_{(g,j)\in A}q_{g,j},\\
 D_B&=\prod_{(g,j)\in B}r_{g,j},&
 X_B&=\prod_{(g,j)\in B}p_{g,j}
\end{aligned}
\]
and
\begin{equation}
 K_\nu(\boldsymbol p_b,\boldsymbol q_b,\boldsymbol r_b)
       =-(-1)^{|A|+|B|}
          \mathcal K_A(D_A,X_A)\mathcal K_B(D_B,X_B).
 \label{eq:g14}
\end{equation}
The family thus consists of the raw member $0$ and all these
comparisons. Every comparison has an auxiliary prime in each
block. Its coefficient uses only the freed auxiliary labels,
the unshared target labels in $A$ and the unshared source labels
in $B$. In particular it ignores the shared labels, the small-bank
labels and the distinguished last labels.

\paragraph{Independent cancellation.}
For fixed $E_{\rm id}>0$ and $C_4\ge0$,
\cite[Theorem~4.1]{OpenAIGraph} supplies $m$, the
parameters $Q,\delta,\xi$ and a constant $C_1$, so that for every sufficiently
large fixed $M$ the family just described satisfies \eqref{eq:g2} and
\begin{equation}
 \sup_{\theta\in\mathbb R,\ |\zeta|\le L^{C_4}}
                        \|T_{\theta,\zeta}\|_{2\to2}\le L^{-E_{\rm id}}.
 \label{eq:imported-ideal}
\end{equation}
Notice the family is independent of $\theta$ and the
individual value of $\zeta$, with its defining parameters fixed
independently of $M$; increasing $M$ supplies additional shared slots.

\paragraph{Physical pairing and absolute bounds.}
For a prescribed $E>0$, choosing $E_{\rm id}$ as in
\cite[Theorem~3.5]{OpenAIGraph} it then follows from (\ref{eq:imported-ideal}) and
\cite[Corollary~3.11]{OpenAIGraph} that, uniformly for
$|\zeta|\le L^{C_4}$,
\begin{equation}
 |\langle f,A_{h,\zeta}^{\rm rep}g\rangle|
                    \ll XL^{-E+2C_3}
 \label{eq:represented-pairing}
\end{equation}
whenever $f(n,\boldsymbol p)=f_n$ is supported on integer
coordinates $n\in I=[X,2X)$ with
$x^\gamma\le X\le x^{1/\gamma}$ and
\[
 \sup|f_n|\le\exp(C\sqrt L),\qquad
 \|f\|_2,\|g\|_2\le X^{1/2}L^{C_3}.
\]
Here $C,C_3,\gamma>0$ are fixed before the graph parameters.
Only the first endpoint must be independent of its prime lists.
The fixed $h$ satisfies the permitted restriction
$0<|h|\le L^{C_2}$ for every fixed $C_2>0$ and large $x$.

It will also be helpful to have control over the absolute values of the edge weights. Let
\[
 (\mathcal B F)(v)
   =\sum_{\nu\in\mathfrak F}\sum_{e:v\to w}
                                  |w_{\nu,0}(e)|F(w).
\]
Notice the absolute value is taken for each labelled edge before adding
patterns or averaging list orders. By
\cite[Lemma~3.4]{OpenAIGraph}, every row sum of $\mathcal B$
and of its adjoint is at most $L^{C_{\rm abs}}$. The exponent $C_{\rm abs}$ depends only on $m,C_1,c_0$ and
the fixed constants in \eqref{eq:banks}, but is independent
of $M$ and the particular choice of active banks. Schur's test bounds the operator norm by the
same power of $L$. The permutation averages preserve these bounds.
They hold for every real $\zeta$ because $|D^{i\zeta}|=1$.

\subsection{Separation of the structured coefficients}

The imported construction gives cancellation for the specific
family just defined. We now prove directly from its coefficients
the separation needed in Section~\ref{sec:arithmetic-bridge}.

\begin{lemma}
\label{lem:separated-comparison-coefficients}
Let $\bigl((C_g^{(\nu)})_g,K_\nu\bigr)_{\nu\in\mathfrak F}$
be the signed comparison family constructed above.
For each comparison $\nu=(A,B)\in\mathfrak F\setminus\{0\}$,
write $\boldsymbol d_1=(r_{g,j})_{(g,j)\in A}$
and $\boldsymbol d_2=(r_{g,j})_{(g,j)\in B}$
for its two auxiliary tuples. There are finite index sets $\mathcal I_\nu$
and real coefficients $c_{\nu j}$ such that
\begin{equation}
 K_\nu=\sum_{j\in\mathcal I_\nu} c_{\nu j}
       a_{\nu j}(\boldsymbol d_1)b_{\nu j}(\boldsymbol d_2)
       u_{\nu j}(\boldsymbol p)v_{\nu j}(\boldsymbol q),
 \label{eq:separated-comparison-coefficients}
\end{equation}
where the four factors $a_{\nu j},b_{\nu j},u_{\nu j},v_{\nu j}$
have modulus at most one and $u_{\nu j}$ and $v_{\nu j}$
depend only on the unshared big-bank labels
of their variables. Furthermore, there exists $C$ independent of $M$ such that
\[
 \sum_{\nu\ne0}\left(|\mathcal I_\nu|
                   +\sum_{j\in\mathcal I_\nu}|c_{\nu j}|\right)
                                                        \le L^C.
\]
\end{lemma}
\begin{proof}
By \eqref{12}, we have $\sum_k\nu_{\Lambda k}=1$, so at most
$\xi^{-1}$ cells satisfy $\nu_{\Lambda k}\ge\xi$. There are at
most $Q^2$ characters in $\mathfrak X_Q$. Thus, by
\eqref{eq:separated-kernel-formula}, each kernel
$\mathcal K_\Lambda$ is a sum of at most $Q^2\xi^{-1}$ products
of two factors of modulus at most one, depending respectively
on $D$ and $X$, with scalars bounded by $\xi^{-1}$. 
Multiplying the two kernels in \eqref{eq:g14} gives the required
separation: the factors involving $D_A,D_B$ depend only on the
respective auxiliary tuples, while those involving $X_B,X_A$
depend only on the unshared source and target labels, respectively. Notice the resulting coefficients $c_{\nu j}$ are real, since they are
products of $\nu_{Ak}^{-1}$, $\nu_{B\ell}^{-1}$
and the real sign $-(-1)^{|A|+|B|}$.

Moreover,
\[
 |\mathcal I_\nu|\le Q^4\xi^{-2},
 \qquad
 \sum_{j\in\mathcal I_\nu}|c_{\nu j}|\le Q^4\xi^{-4}.
\]
There are at most
\[
 2^{m|\mathcal G_b|}\le L^{mC_*\log 2}
\]
comparison patterns, using the inherited upper bound in
\eqref{eq:banks}. Since $m$ is fixed and $Q,\xi^{-1}$ are bounded
by fixed powers of $L$, the asserted bound follows with an exponent
independent of $M$ and the particular choice of banks.
\end{proof}

\subsection{Complete lists and the exact pairing}

We need vertices whose prime lists are all complete. Define
\[
 \Omega_M=\{(n,\boldsymbol p)\in\Omega_M^{\rm rep}:
                         \omega_g(n)=M\text{ for every }g\},
 \qquad \mathcal H_M=\ell^2(\Omega_M,\lambda_M).
\]
Every list at a vertex of $\Omega_M$ orders all its bank prime
divisors. We identify $\mathcal H_M$ with its zero extension in
$\mathcal H_M^{\rm rep}$ and use the same vertex measure for
all its $\ell^r$ norms. The orthogonal projection onto this space is
\[
 (P_Mf)(n,\boldsymbol p)
     =\left(\prod_g\one_{\{\omega_g(n)=M\}}\right)f(n,\boldsymbol p).
\]

Restrict each represented graph to $\Omega_M$, retaining the
same weights and precisely the edges
\[
 E_{h,\nu}^{M}=\{(v,w,\boldsymbol r)\in E_{h,\nu}:
                                      v,w\in\Omega_M\}.
\]
We call these the \emph{exact-count graphs}. Their signed
adjacency operator, with the same list-permutation averages, is
\begin{equation}
                         A_{h,\zeta}=P_M A_{h,\zeta}^{\rm rep}P_M.
 \label{eq:exact-compression}
\end{equation}
Thus the exact-count restriction is imposed on both endpoints.

\begin{theorem}[The exact graph pairing]
\label{thm:finite-graph}
Fix $E_{\rm tar}>0$, $C_4\ge0$, $\gamma>0$,
$C_3,C_r>0$ and $r>2$. Then the structured
comparison family above can be chosen, with
$m,C_1$ and the powers of $L$ governing
$Q,\delta^{-1},\xi^{-1}$ fixed before choosing
a sufficiently large fixed $M$, so that the
following holds. Suppose $f,g\in\mathcal H_M$ satisfy
\[
 \|f\|_2,\|g\|_2\le X^{1/2}L^{C_3},\qquad
                    \|f\|_r^r\ll XL^{C_r}.
\]
If $f(n,\boldsymbol p)=f_n$ is supported on $I=[X,2X)$ with
$x^\gamma\le X\le x^{1/\gamma}$, then, uniformly for
$|\zeta|\le L^{C_4}$,
\begin{equation}
                 |\langle f,A_{h,\zeta}g\rangle|
                         \ll XL^{-E_{\rm tar}+2C_3}.
 \label{eq:g6}
\end{equation}
The labelled absolute row and column bounds for $A_{h,\zeta}$ are
$L^{C_{\rm abs}}$, with exponent independent of $M$.
\end{theorem}
\begin{proof}
First check what the compression retains. On both endpoints of
$P_M A_{h,\zeta}^{\rm rep}P_M$, every represented list is complete and
both damping factors in \eqref{eq:represented-edge-weight} are one.
The edges retain exactly the previously defined shared slots,
auxiliary exclusions and edge weights
\eqref{eq:represented-edge-weight}. Since $P_M$ depends only on the
integer, it commutes with the list-permutation projection $\Pi$.
Consequently $A_{h,\zeta}$ is precisely the signed adjacency operator on
complete-list vertices with these weights and permutation averages.
No extra factors are introduced by restriction. Removing vertices
and edges from the positive absolute operator can only decrease
its row and column sums, proving the absolute bounds.

Extend $f,g$ by zero to the represented space. This is isometric
for every $\ell^p$ norm. Notice that at a given integer, either every represented
list belongs to the exact set or none does, so $f(n,\boldsymbol p)$ remains independent of its marks.
For $B=\exp(\sqrt L)$, define the clipped vector by
\[
 f^{(B)}(n,\boldsymbol p)=
 \begin{cases}
 f_n,&|f_n|\le B,\\
 Bf_n/|f_n|,&|f_n|>B.
 \end{cases}
\]
Its discarded part $f^{\rm tail}=f-f^{(B)}$ satisfies
\[
 \|f^{\rm tail}\|_2^2\le B^{-(r-2)}\|f\|_r^r.
\]
The represented absolute bound and Schur's test give
\[
 |\langle f^{\rm tail},A_{h,\zeta}^{\rm rep}g\rangle|
 \ll XL^{C_{\rm abs}+C_3+C_r/2}
                         \exp(-(r-2)\sqrt L/2).
\]
This is negligible to every fixed logarithmic order. The clipped vector still depends only on $n$. Its $\ell^2$
norm does not increase. Apply \eqref{eq:represented-pairing}, choosing
$E>E_{\rm tar}$ and the parameters in the order stated there.
By \eqref{eq:exact-compression} this proves \eqref{eq:g6}.
\end{proof}

\subsection{Deduction of the graph comparison}
\label{sec:graph-comparison-deduction}

We now insert the harmonic smooth weight in \eqref{eq:g6} and
remove the shared prime labels. This will prove
Proposition~\ref{prop:graph-comparison} with exactly the comparison
terms specified in \eqref{eq:graph-comparison-term}.

\paragraph{Inserting the harmonic weight.}
Fix $A>0$ and $\Psi\in C_c^\infty((0,\infty))$. Let $C_{\rm end}$ be a sufficiently large constant, independent of $M$, to be specified soon.
Apply Theorem~\ref{thm:finite-graph} with
\[
 C_4=1,\qquad\gamma=1/2,\qquad
                    E_{\rm tar}>A+2C_{\rm end}+3
\]
and choose one admissible fixed $M$. The pattern family is the same
for every $|\zeta|\le L$ and its coefficient exponents have already
been chosen before $M$.

For an edge of a pattern, define its shared and auxiliary products by
\[
 (D_R)_g=\prod_{\substack{j<M\\j\notin C_g}}p_{g,j},\qquad
 D_R=\prod_g(D_R)_g,\qquad
 D_C=\prod_g\prod_{j\in C_g}r_{g,j}.
\]
Then the displacement product $D$ defined in
\eqref{eq:graph-displacement-product} factors as
$D=D_RD_C$. Writing $n=n(e)$, $n'=n'(e)$ and $D=D(e)$
inside the edge sum, consider
\begin{equation}
 \mathcal C_\Psi
   =\lambda_M\sum_{\nu\in\mathfrak F}
      \sum_{\substack{e\in E_{h,\nu}^{M}\\n>0}}
       w_{\nu,0}(e)\,
        \alpha_{\mathcal P}(n)\beta_{\mathcal P}(n')
        \frac{\Psi(n/(xD))}{n},
 \label{eq:graph-weighted-sum}
\end{equation}
where $w_{\nu,0}(e)$ are the edge-weights introduced in
\eqref{eq:represented-edge-weight}, with the damping factors $\delta_M$ now equal to $1$ since we are restricting to edges in $E_{h,\nu}^{M}$.
Let $\operatorname{supp}\Psi\subset[c_\Psi,C_\Psi]$ with
$0<c_\Psi<C_\Psi$ and choose a common bound
\[
 D\le D_{\max}:=\exp(c_MJL^{2/5})=x^{o(1)},
\]
where $c_M>0$ is a sufficiently large constant depending on the
fixed $M$ and bank bounds. Every source integer in
\eqref{eq:graph-weighted-sum} with nonzero weight belongs to
$\mathcal J_x=[c_\Psi x,C_\Psi xD_{\max}]$.
Put $\varphi(u)=\Psi(e^u)$ and
$\widehat\varphi(\zeta)=\int\varphi(u)e^{-i\zeta u}\,du$.
For $n>0$, Fourier inversion gives
\begin{equation}
 \frac{\Psi(n/(xD))}{n}
   =\frac1{2\pi n}\int_{\mathbb R}
       \widehat\varphi(\zeta)(n/x)^{i\zeta}D^{-i\zeta}\,d\zeta.
 \label{eq:graph-harmonic-inversion}
\end{equation}
Cover $\mathcal J_x$ by disjoint dyadic intervals $I=[X,2X)$,
using $O(L)$ intervals with $X=x^{1+o(1)}$. On one such dyadic interval, use $A_{h,-\zeta}$ with endpoint functions
\[
\begin{aligned}
 f(n,\boldsymbol p)=f_n
   &=\overline{\alpha_{\mathcal P}(n)
          \one_{I\cap\mathcal J_x}(n)(X/n)(n/x)^{i\zeta}},\\
 g(n,\boldsymbol p)=g_n
   &=\beta_{\mathcal P}(n)\one_{[X/2,3X]}(n).
\end{aligned}
\]
Since $|hD|\le|h|D_{\max}=x^{o(1)}=o(X)$, this target cutoff
contains every edge leaving the source support. Here we conjugate
$f_n$ due to our inner-product convention. Inserting \eqref{eq:graph-harmonic-inversion}
into \eqref{eq:graph-weighted-sum}, we see that
the contribution to $\mathcal C_\Psi$ from
edges whose source integer lies in $I$ is
\begin{equation}
\label{int}
 \frac1{2\pi X}\int_{\mathbb R}
 \widehat\varphi(\zeta)
 \langle f,A_{h,-\zeta}g\rangle\,d\zeta.
\end{equation}
Here the permutation projections do not affect the sum,
because both endpoint functions are independent
of the ordering of their prime lists. By Lemma~\ref{lem:marked-moments}, we have $\|f\|_2,\|g\|_2\le X^{1/2}L^{C_{\rm end}}$ and $ \|f\|_4^4\ll XL^{C_{\rm end}}$, provided $C_{\rm end}$ was chosen sufficiently large. This gives us the required
moment conditions in Theorem~\ref{thm:finite-graph}, with $r=4$. Hence
\[
 X^{-1}|\langle f,A_{h,-\zeta}g\rangle|
       \ll L^{-E_{\rm tar}+2C_{\rm end}}\ll L^{-A-3}
                               \qquad(|\zeta|\le L).
\]
Inserting this in (\ref{int}) controls the range $| \zeta | \le L$ of the integral by $O_A (L^{-A-2})$. For the remaining range of $\zeta$, the absolute row and column
bounds, together with Schur's test and
Cauchy--Schwarz, already give
\[
 X^{-1}|\langle f,A_{h,-\zeta}g\rangle|
 \le X^{-1}\|f\|_2
       \|A_{h,-\zeta}\|_{2\to2}\|g\|_2
 \ll L^{C_{\rm abs}+2C_{\rm end}}.
\]
Choosing a sufficiently large decay
order in the Schwartz bound for $\widehat\varphi$, we see that (\ref{int}) is $O(L^{-A-2})$ when restricted to $|\zeta|> L$. Summing over all the dyadic intervals we conclude that \begin{equation}
                  \mathcal C_\Psi=O_A(L^{-A}).
 \label{eq:graph-weighted-bound}
\end{equation}

\paragraph{Removing the shared labels.}
Our remaining task is to recover
the arithmetic expressions in
Proposition~\ref{prop:graph-comparison} from $\mathcal C_{\Psi}$. With this in mind, we wish to remove the common dilations given by the shared products $D_R$. Let us fix a pattern $\nu\in\mathfrak F$ and write
$C_g=C_g^{(\nu)}$, $t_g=1+|C_g|$ and $K=K_{\nu}$. Recall that $K$ does not depend on the shared labels, which follows from
\eqref{eq:g14} for a comparison and from $K_0=1$ for the raw
member. Write
\[
 n=D_Rw,\qquad n+hD=D_Rw',\qquad w'=w+hD_C.
\]
The shared primes on an original edge are distinct. They satisfy
\begin{equation}
\begin{aligned}
 \omega_g(D_Rw)=M
 &\iff\omega_g(w)=t_g
 &&\text{if }(w,D_R)=1,\\
 \frac{V_g^{-M-t_g}}{(D_R)_g}
 &=V_g^{-2t_g}
   \prod_{\substack{j<M\\j\notin C_g}}\mu_g(p_{g,j}).
\end{aligned}
\label{eq:bridge-shared-identities}
\end{equation}
and
\[
\begin{aligned}
 \alpha_{\mathcal P}(D_Rw)&=\alpha_{\mathcal P}(w),&
 \beta_{\mathcal P}(D_Rw')&=\beta_{\mathcal P}(w'),\\
 \frac{\Psi(n/(xD))}{n}
       &=\frac1{D_R}\frac{\Psi(w/(xD_C))}{w}.&&
\end{aligned}
\]
where we have used \eqref{eq:core-def}. Combining the last four identities with
\eqref{eq:represented-edge-weight} and recalling that here the damping
factors equal one, we obtain, for every original edge,
\begin{equation}
\begin{aligned}
\label{something}
\lambda_M w_{\nu,0}(e)
  \alpha_{\mathcal P}(n)\beta_{\mathcal P}(n')
  \frac{\Psi(n/(xD))}{n}
&=
  \frac{\Psi(w/(xD_C))}{w}
  \alpha_{\mathcal P}(w)\beta_{\mathcal P}(w')
  K_\nu(\boldsymbol p_b,\boldsymbol q_b,\boldsymbol r_b)\\
&\qquad\quad\times
  \prod_{g\in\mathcal G}
  \left(
  V_g^{-2t_g}
  \prod_{\substack{j<M\\j\notin C_g}}
       \mu_g(p_{g,j})
  \prod_{j\in C_g}\mu_g(r_{g,j})
  \right).
\end{aligned}
\end{equation}

\paragraph{Collision errors.}
Notice that the shared primes still remain on the right-hand
side through their probability weights $\mu_g(p_{g,j})$.
If for each fixed choice of $w, w'$ these shared primes were drawn randomly, these probabilities would sum to $1$ and thus disappear.
However, only choices that reconstruct an admissible edge are
allowed. In particular, the shared primes must be distinct and
avoid the auxiliary primes. We now show that the valid choices
have total probability close to one when the reduced endpoints
have the required bank counts and negligible probability otherwise.

Fix $w>0$, $w'=w+hD_C$, the auxiliary primes and, in every bank $g$,
ordered lists of $t_g$ distinct primes dividing $w$ and $w'$
respectively. Retain for now the edge requirement that the auxiliary
primes avoid these unshared lists. We may restrict to pairs satisfying $
 t_g\le\omega_g(w),\omega_g(w')\le M$,
since clearly every pair arising from an original graph edge satisfies these
inequalities. Notice however that $ \omega_g(w)=\omega_g(w')=t_g$ need not hold for an original edge, since
a shared prime may divide $w$ or $w'$.

For these fixed data, let
\[
 S=\sum_{\substack{\text{valid choices of}\\\text{shared primes}}}
       \prod_g\prod_{\substack{j<M\\j\notin C_g}}
                                     \mu_g(p_{g,j}),
\]
where a choice is \emph{valid} if it reconstructs an edge of the
original exact-count graph, with endpoints $n=D_Rw$ and
$n'=D_Rw'$. Draw the shared primes independently according to
their laws $\mu_g$. Their unrestricted probability weights sum
to $1$, so $0\le S\le1$.

There are $O_M(J)$ shared-prime draws. On the support of the
smooth weight, $w,w'=x^{1+o(1)}$ and $D_C=x^{o(1)}$, so at most \(O(L)\) prime values either divide \(ww'D_C\) or have already been chosen as shared primes. By \eqref{eq:banks},
each individual prime has probability
\[
 \mu_g(p)=\frac1{V_gp}\le 2e^{-L^{1/8}}.
\]
A union bound therefore shows that the probability of a repeated
shared prime or a shared prime dividing $ww'D_C$ is at most
$\varepsilon_L$, where, uniformly in the fixed data,
\[
 \varepsilon_L\ll L^{O_M(1)}e^{-L^{1/8}}.
\]
We now distinguish two cases.

\emph{First, suppose $\omega_g(w)=\omega_g(w')=t_g$ in every bank.}
Every choice of distinct shared primes avoiding $ww'D_C$ is valid.
Indeed, adjoining the $M-t_g$ shared primes in bank $g$ gives
exactly $M$ distinct prime divisors at each endpoint. Avoidance
of $D_C$ ensures the remaining auxiliary-prime restrictions.
These choices have total probability at least $1-\varepsilon_L$.
Thus $S=1+O(\varepsilon_L)$.

\emph{Second, suppose the counts are not all equal to $t_g$.}
Some reduced endpoint then has more than $t_g$ prime divisors in
some bank. Suppose for example $\omega_g(w)>t_g$. A valid choice
must contain $M-t_g$ distinct shared primes in that bank. If
none divided $w$, we would have $ \omega_g(D_Rw)=\omega_g(w)+(M-t_g)>M$,
contrary to the exact-count condition at $n$. Hence every valid
choice includes a shared prime dividing $w$. The same argument
applies if the excess count occurs at $w'$. The probability of
selecting a shared prime dividing $ww'$ is at most
$\varepsilon_L$, so $
 0\le S\le\varepsilon_L$.
 
Combining the two cases gives
\[
 S=\prod_g\one_{\{\omega_g(w)=\omega_g(w')=t_g\}}
       +O(\varepsilon_L).
\]
Let $H$ denote the product of all factors in the right hand side of (\ref{something})
 except the shared-prime probability weights.
These factors are independent of the shared primes, so the
original contribution for our fixed data is $HS$. Therefore
\[
 HS=H\prod_g\one_{\{\omega_g(w)=\omega_g(w')=t_g\}}
          +O\bigl(|H|\varepsilon_L\bigr).
\]
We now sum these errors over our previously fixed choices. First, for fixed $w,w'$, the sum over all pairs of
ordered unshared lists, with their corresponding normalizations $V_g^{-2t_g}$, costs at most a factor of 
\[
 \prod_g(M)_{t_g}^2V_g^{-2t_g}\le L^{O_M(1)}.
\]
For fixed auxiliary primes, we now sum over $w$, with $w'=w+hD_C$. Put $Y=xD_C=x^{1+o(1)}$. On the
support of $\Psi(w/Y)$, both $w$ and $w'$ are comparable to $Y$.
Cauchy--Schwarz and the second divisor-moment bounds give
\[
\begin{aligned}
 \sum_{w\asymp Y}
   \frac{|\alpha_{\mathcal P}(w)\beta_{\mathcal P}(w')|}{w}
 &\ll \frac1Y
 \left(\sum_{w\asymp Y}|\alpha_{\mathcal P}(w)|^2\right)^{1/2}
 \left(\sum_{w\asymp Y}|\beta_{\mathcal P}(w')|^2\right)^{1/2}\\
 &\ll L^{O(1)}.
\end{aligned}
\]
Finally, the sum of $\prod_g\prod_{j\in C_g}\mu_g(r_{g,j})$ is clearly bounded by $1$, since $\mu_g$ are probability measures. Combining all these estimates with \eqref{eq:g2} bounds
the total error of replacing $HS$ with  $H\prod_g\one_{\{\omega_g(w)=\omega_g(w')=t_g\}}$ by
\[
 L^{O_M(1)}e^{-L^{1/8}}\ll_A L^{-A}
\]
for every fixed $A>0$. In particular, the resulting expression no longer depends on shared primes.

It remains to remove the restrictions requiring
the auxiliary primes to avoid the unshared endpoint
lists. Since these lists are now complete, an
auxiliary prime $p$ occurs in one of them precisely
when it divides the corresponding reduced endpoint.
As $p\mid D_C$, we have
\[
 p\mid w\quad\Longleftrightarrow\quad
 p\mid w+hD_C.
\]
At scale $Y\asymp xD_C$, Cauchy--Schwarz
and the fourth divisor-moment bounds give
\[
 \frac1Y\sum_{\substack{w\asymp Y\\p\mid w}}
 |\alpha_{\mathcal P}(w)
   \beta_{\mathcal P}(w+hD_C)|
 \ll L^{O(1)}p^{-1/2}.
\]
There are only $O_m(J)$ auxiliary labels, each
at least $e^{L^{1/3}}$, so the total contribution
of these excluded configurations is negligible.
The same estimate covers any prime appearing
in both unshared endpoint lists: such a prime
divides $hD_C$, and since the bank primes avoid
$h$, it must be one of the auxiliary primes. All these errors remain negligible after including
the polynomial total coefficient cost of the
comparison family given by \eqref{eq:g2}. 

\paragraph{Reading off the arithmetic relation.}
Write $\mathcal C_{\Psi,\mathrm{raw}}$ for the
contribution of the raw pattern to $\mathcal C_\Psi$.
For this pattern, $D_C=1$, $\boldsymbol t=\boldsymbol1$
and $K_0=1$. It thus follows from
\eqref{something} and the preceding analysis that
\[
\begin{aligned}
\mathcal C_{\Psi,\mathrm{raw}}
={}&\sum_{w\ge1}\frac{\Psi(w/x)}w
\alpha_{\mathcal P}(w)\beta_{\mathcal P}(w+h)\\
&\quad\times
\underbrace{\prod_{g\in\mathcal G}
\frac{\mathbf1_{\{\omega_g(w)=\omega_g(w+h)=1\}}}{V_g^2}}_{W_{\boldsymbol1}(w)W_{\boldsymbol1}(w+h)}
+O_A(L^{-A})\\
={}&\mathcal C_0+O_A(L^{-A}),
\end{aligned}
\]
 where the weights $W_{\boldsymbol t}^b$ were defined
in \eqref{eq:bridge-mark-weight}. For a comparison $\nu=(A,B)$, put $D_1=D_A$ and $D_2=D_B$.
The definitions give $D_C=D_1D_2$. Expand its coefficient $K$ in \eqref{something} using
Lemma~\ref{lem:separated-comparison-coefficients} and write $\mathcal I_\nu$ for the index set provided by the statement of that lemma. Relabel the
unshared slots in each bank in increasing order as $1,\ldots,t_g$.
Since
\[
\begin{aligned}
&\left(\prod_g
\frac{\mathbf1_{\{\omega_g(w)=\omega_g(w')=t_g\}}}
{V_g^{2t_g}}\right)
\sum_{\substack{\boldsymbol p\text{ a }\boldsymbol t\text{-list}\\\text{at }w}}
\sum_{\substack{\boldsymbol q\text{ a }\boldsymbol t\text{-list}\\\text{at }w'}}
u_{\nu j}(\boldsymbol p)v_{\nu j}(\boldsymbol q)\\
&\hspace{12mm}=
W_{\boldsymbol t}^{u_{\nu j}}(w)
W_{\boldsymbol t}^{v_{\nu j}}(w'),
\end{aligned}
\]
we see that treating each pair $(\nu,j)$ with
$j\in\mathcal I_\nu$ as a separate comparison index,
we recover \eqref{eq:graph-comparison-term}. By Lemma~\ref{lem:separated-comparison-coefficients}, we have the second inequality in \eqref{eq:graph-comparison-relation}. By construction, each comparison has at least one auxiliary
prime from each block and at most $m$ from any bank. The estimate \eqref{eq:graph-weighted-bound} therefore gives
\[
             \mathcal C_0+\sum_\nu c_\nu\mathcal C_\nu
                                      =O_A(L^{-A}).
\]
This concludes the proof of Proposition~\ref{prop:graph-comparison} and thus of Theorem~\ref{thm:main}.

\phantomsection

\bigskip
\begingroup\small
\noindent\textsc{Departamento de Matem\'atica e IMAS-CONICET,
Facultad de Ciencias Exactas y Naturales, Universidad de Buenos Aires,
1428 Buenos Aires, Argentina}\par
\smallskip
\noindent\textit{Email address:}
\href{mailto:mwalsh@dm.uba.ar}{\texttt{mwalsh@dm.uba.ar}}\par
\endgroup
\end{document}